\documentclass{article}
\usepackage{array,amsmath,amsthm}
\numberwithin{equation}{section}
\usepackage{amssymb}
\usepackage{indentfirst}
\usepackage{geometry}
\begin{document}

\newtheorem{thm}{Theorem}[section]
\newtheorem{cor}[thm]{Corollary}
\newtheorem{prop}[thm]{Proposition}
\newtheorem{conj}[thm]{Conjecture}
\newtheorem{lem}[thm]{Lemma}
\newtheorem{Def}[thm]{Definition}
\newtheorem{rem}[thm]{Remark}
\newtheorem{prob}[thm]{Problem}
\newtheorem{ex}{Example}[section]

\newcommand{\be}{\begin{equation}}
\newcommand{\ee}{\end{equation}}
\newcommand{\ben}{\begin{enumerate}}
\newcommand{\een}{\end{enumerate}}
\newcommand{\beq}{\begin{eqnarray}}
\newcommand{\eeq}{\end{eqnarray}}
\newcommand{\beqn}{\begin{eqnarray*}}
\newcommand{\eeqn}{\end{eqnarray*}}
\newcommand{\bei}{\begin{itemize}}
\newcommand{\eei}{\end{itemize}}

\newcommand{\pa}{{\partial}}
\newcommand{\V}{{\rm V}}
\newcommand{\R}{{\bf R}}
\newcommand{\K}{{\rm K}}
\newcommand{\e}{{\epsilon}}
\newcommand{\tomega}{\tilde{\omega}}
\newcommand{\tOmega}{\tilde{Omega}}
\newcommand{\tR}{\tilde{R}}
\newcommand{\tB}{\tilde{B}}
\newcommand{\tGamma}{\tilde{\Gamma}}
\newcommand{\fa}{f_{\alpha}}
\newcommand{\fb}{f_{\beta}}
\newcommand{\faa}{f_{\alpha\alpha}}
\newcommand{\faaa}{f_{\alpha\alpha\alpha}}
\newcommand{\fab}{f_{\alpha\beta}}
\newcommand{\fabb}{f_{\alpha\beta\beta}}
\newcommand{\fbb}{f_{\beta\beta}}
\newcommand{\fbbb}{f_{\beta\beta\beta}}
\newcommand{\faab}{f_{\alpha\alpha\beta}}

\newcommand{\pxi}{ {\pa \over \pa x^i}}
\newcommand{\pxj}{ {\pa \over \pa x^j}}
\newcommand{\pxk}{ {\pa \over \pa x^k}}
\newcommand{\pyi}{ {\pa \over \pa y^i}}
\newcommand{\pyj}{ {\pa \over \pa y^j}}
\newcommand{\pyk}{ {\pa \over \pa y^k}}
\newcommand{\dxi}{{\delta \over \delta x^i}}
\newcommand{\dxj}{{\delta \over \delta x^j}}
\newcommand{\dxk}{{\delta \over \delta x^k}}

\newcommand{\px}{{\pa \over \pa x}}
\newcommand{\py}{{\pa \over \pa y}}
\newcommand{\pt}{{\pa \over \pa t}}
\newcommand{\ps}{{\pa \over \pa s}}
\newcommand{\pvi}{{\pa \over \pa v^i}}
\newcommand{\ty}{\tilde{y}}
\newcommand{\bGamma}{\bar{\Gamma}}

\title {On Finsler metric measure manifolds with integral weighted Ricci curvature bounds\footnote{The first author is supported by the National Natural Science Foundation of China (12371051, 12141101, 11871126). The second author is supported by the Chongqing Postgraduate Research and Innovation Project (CYB23231).}}
\author{ Xinyue Cheng$^{1}$ \& Yalu Feng$^{1}$\\
$^{1}$ School of Mathematical Sciences, Chongqing Normal University, \\
Chongqing, 401331, P.R. China\\
E-mails: chengxy@cqnu.edu.cn (X. Cheng), fengyl2824@qq.com (Y. Feng)}
\date{}

\maketitle

\begin{abstract}
In this paper, we study deeply geometric and topological properties of Finsler metric measure manifolds with the integral weighted Ricci curvature bounds. We first establish Laplacian comparison theorem, Bishop-Gromov type volume comparison theorem and relative volume comparison theorem on such Finsler manifolds. Then we obtain a volume growth estimate and Gromov pre-compactness under the integral weighted Ricci curvature bounds. Furthermore, we prove the local Dirichlet isoperimetric constant estimate on Finsler metric measure manifolds with integral weighted Ricci curvature bounds. As applications of the Dirichlet isoperimetric constant estimates, we get first Dirichlet eigenvalue estimate and a gradient estimate for harmonic functions.\\
{\bf Keywords:} Finsler metric measure manifold; integral weighted Ricci curvature; Gromov pre-compactness; isoperimetric constant; Sobolev constant; gradient estimate\\
{\bf Mathematics Subject Classification:} 53C60,  53B40, 58C35

\end{abstract}

\section{Introduction}
Integral Ricci curvature in Riemannian geometry is a very important notion that naturally appears in diverse situations and has been extensively studied in the last few decades.  In \cite{Ga}, Gallot studied the isoperimetric inequalities by studying inequalities on eigenvalues or Sobolev constants on Riemannian manifolds with integral Ricci curvature bounds.  In the fundamental work \cite{PeterW1}, the important Laplacian comparison and volume comparison are generalized to Riemannian manifolds with integral Ricci curvature lower bound.   Combining this with D. Yang's estimate \cite{Yang1} on the local Sobolev constant, the Cheeger-Colding-Naber theory has now been successfully extended to the manifolds with integral Ricci curvature bounds \cite{PeterW2}. Recently, Dai-Wei-Zhang \cite{DWZ} obtained a local Sobolev constant estimate under the integral Ricci curvature bounds. Based on this, they extended the maximum principal, the gradient estimate, the heat kernel estimate and the $L^2$-Hessian estimate to Riemannian manifolds with integral Ricci curvature lower bounds, without the non-collapsing conditions. Further, Wang-Wei generalized the local Sobolev constant estimate and related results in \cite{DWZ} to complete smooth Riemannian metric measure spaces with integral Bakry-\'{E}mery Ricci curvature bounds \cite{Wang-Wei}.

Finsler geometry is just Riemannian geometry without the quadratic restriction \cite{Chern}. It is natural to study geometric and topological properties of Finsler metric measure manifolds with the integral Ricci curvature bounds.
Wu B.-Y. \cite{WuB2012,WuB2021} obtained the volume comparison theorem and proved finiteness of fundamental group and the Gromov pre-compactness of Finsler manifolds under the integral Ricci curvature bounds. He also characterized volume growth and the relationship between integral Ricci curvature and the first Betti number of Finsler manifolds with integral Ricci curvature bounds. Later, Zhao \cite{ZhaoW} established a relative volume comparison and obtained several Myers type theorems on Finsler manifolds with integral Ricci curvature bounds.

In this paper, following the arguments in \cite{PeterW1, PeterW2}, we introduce the integral weighted Ricci curvature and try to extend some results in \cite{PeterW1, PeterW2} and  \cite{DWZ,Wang-Wei} to Finsler metric measure manifolds with integral weighted Ricci curvature bounds.

A Finsler manifold $(M, F)$ equipped with a measure $m$ is called a Finsler metric measure manifold (or Finsler measure space for short) and is denoted by $(M, F, m)$. The class of Finsler measure spaces is one of the most important metric measure spaces. However, a Finsler metric measure space is not a metric space in the usual sense because a Finsler metric $F$ may be nonreversible, that is, $F(x, y)\neq F(x, -y)$ may happen for $(x,y) \in TM$. This non-reversibility causes the asymmetry of the associated distance function. Moreover, unlike Riemannian Laplacian, Finsler Laplacian is a nonlinear elliptic differential operator of the second order. This is another difference between Finsler metric measure spaces and other metric measure spaces. For more details, please refer to \cite{BaoChSh,ChernShen,ChSh,KristalyZ,OHTA,Ohta,Xia}.

The paper is organized as follows. In Section \ref{Pre}, we give some necessary definitions and notations. Then we first derive a Laplacian comparison theorem (Theorem \ref{phi-laplacian}) and several volume comparison theorems (Theorem \ref{phi-volume}, Theorem \ref{doubling} and Theorem \ref{relative-volume}) with integral weighted Ricci curvature bounds in Section \ref{Integral}.  Based on these, we give a volume growth estimate (Theorem \ref{growth}) and Gromov pre-compactness (Theorem \ref{pre-com1} and Theorem \ref{pre-com2}) of Finsler metric measure manifolds with integral weighted Ricci curvature bounds still in Section \ref{Integral}. Section \ref{Sobolev} is devoted to the proof of the Dirichlet isoperimetric constant estimate (Theorem \ref{constant}). Finally, as applications of the Dirichlet isoperimetric constant estimate, we get the first (nonzero) Dirichlet eigenvalue estimate (Theorem \ref{eiges}) and a gradient estimate (Theorem \ref{gradient}) for harmonic functions in Section \ref{Application}.

\section{Preliminaries}{\label{Pre}}
In this section, we briefly review some necessary definitions, notations and the fundamental results in Finsler geometry. For more details, we refer to \cite{BaoChSh, OHTA, Shen1}.

Let $M$ be an $n$-dimensional smooth manifold. A Finsler metric on manifold $M$ is a function $F: T M \longrightarrow[0, \infty)$  satisfying the following properties: (1) $F$ is $C^{\infty}$ on $TM\backslash\{0\}$; (2) $F(x,\lambda y)=\lambda F(x,y)$ for any $(x,y)\in TM$ and all $\lambda >0$; (3)  $F$ is strongly convex, that is, the matrix $\left(g_{ij}(x,y)\right)=\left(\frac{1}{2}(F^{2})_{y^{i}y^{j}}\right)$ is positive definite for any nonzero $y\in T_{x}M$. The pair $(M,F)$ is called a Finsler manifold and $g:=g_{ij}(x,y)dx^{i}\otimes dx^{j}$ is called the fundamental tensor of $F$.

For a non-vanishing vector field $V$ on $M$, one introduces the weighted Riemannian metric $g_V$ on $M$ given by
\be
g_V(y, w)=g_{ij}(x, V_x)y^i w^j  \label{weiRiem}
\ee
for $y, w\in T_{x}M$. In particular, $g_{V}(V,V)=F^2(x,V)$.

We define the reverse metric $\overleftarrow{F}$ of a Finsler metric $F$ by $\overleftarrow{F}(x, y):=F(x,-y)$ for all $(x, y) \in T M$. It is easy to see that $\overleftarrow{F}$ is also a Finsler metric on $M$. A Finsler metric $F$ on $M$ is said to be reversible if $\overleftarrow{F}(x, y)=F(x, y)$ for all $(x, y) \in T M$. In order to overcome the deficiencies that a Finsler metric $F$ may be nonreversible, Rademacher defined the reversibility $\Lambda_{F}$ of $F$ by
\be
\Lambda_F:=\sup _{(x, y) \in TM \backslash\{0\}} \frac{\overleftarrow{F}(x, y)}{F(x, y)}.
\ee
Obviously, $\Lambda_F \in [1, \infty]$ and $\Lambda_F=1$ if and only if $F$ is reversible \cite{Ra}. On the other hand, Ohta extended the concepts of uniform smoothness and uniform convexity in Banach space theory into Finsler geometry and gave their geometric interpretation in \cite{Ohta}. We say that $F$ satisfies uniform convexity and uniform smoothness if there exist two uniform constants $0<\kappa^{*}\leq 1 \leq \kappa <\infty$ such that for $x\in M$, $V\in T_{x}M\setminus \{0\}$ and $W\in T_{x}M$,
\begin{equation}
\kappa^{*}F^{2}(x, W)\leq g_{V}(W, W)\leq \kappa F^{2}(x, W). \label{usk}
\end{equation}
If $F$ satisfies the uniform smoothness and uniform convexity, then $\Lambda_F$ is finite with
$$
1 \leq \Lambda_F \leq \min \left\{\sqrt{\kappa}, \sqrt{1 / \kappa^*}\right\}.
$$
$F$ is Riemannian if and only if $\kappa=1$ if and only if $\kappa^*=1$ \cite{Ohta,Ra}.

Let $(M,F)$ be a Finsler manifold of dimension $n$. The pull-back $\pi ^{*}TM$ admits a unique linear connection, which is called the Chern connection. The Chern connection $D$ is determined by the following equations
\beq
&& D^{V}_{X}Y-D^{V}_{Y}X=[X,Y], \label{chern1}\\
&& Zg_{V}(X,Y)=g_{V}(D^{V}_{Z}X,Y)+g_{V}(X,D^{V}_{Z}Y)+ 2C_{V}(D^{V}_{Z}V,X,Y) \label{chern2}
\eeq
for $V\in TM\setminus \{0\}$  and $X, Y, Z \in TM$, where
$$
C_{V}(X,Y,Z):=C_{ijk}(x,V)X^{i}Y^{j}Z^{k}=\frac{1}{4}\frac{\pa ^{3}F^{2}(x,V)}{\pa V^{i}\pa V^{j}\pa V^{k}}X^{i}Y^{j}Z^{k}
$$
is the Cartan tensor of $F$ and $D^{V}_{X}Y$ is the covariant derivative with respect to the reference vector $V$.

Given a non-vanishing vector field $V$ on $M$,  the Riemannian curvature  $R^V$ is defined by
$$
R^V(X, Y) Z=D_X^V D_Y^V Z-D_Y^V D_X^V Z-D_{[X, Y]}^V Z
$$
for any vector fields $X$, $Y$, $Z$ on $M$. For two linearly independent vectors $V, W \in T_x M \backslash\{0\}$, the flag curvature is defined by
$$
\mathcal{K}^V(V, W)=\frac{g_V\left(R^V(V, W) W, V\right)}{g_V(V, V) g_V(W, W)-g_V(V, W)^2}.
$$
Then the Ricci curvature is defined as
$$
\operatorname{Ric}(V):=F(x, V)^{2} \sum_{i=1}^{n-1} \mathcal{K}^V\left(V, e_i\right),
$$
where $e_1, \ldots, e_{n-1}, \frac{V}{F(V)}$ form an orthonormal basis of $T_x M$ with respect to $g_V$.

For $x_1, x_2 \in M$, the distance from $x_1$ to $x_2$ is defined by
$$
d\left(x_1, x_2\right):=\inf _\gamma \int_0^1 F(\gamma(t), \dot{\gamma}(t)) d t,
$$
where the infimum is taken over all $C^1$ curves $\gamma:[0,1] \rightarrow M$ such that $\gamma(0)=$ $x_1$ and $\gamma(1)=x_2$. Note that $d \left(x_1, x_2\right) \neq d \left(x_2, x_1\right)$ unless $F$ is reversible. The forward geodesic balls and backward geodesic balls of radius $R$ with center at $x_{0}$ are defined by
$$
B_R^{+}(x_0):=\{x \in M \mid d(x_{0}, x)<R\},\quad B_R^{-}(x_0):=\{x \in M \mid d(x, x_{0})<R\}.
$$

A $C^{\infty}$-curve $\gamma:[0,1] \rightarrow M$ is called a geodesic  if $F(\gamma, \dot{\gamma})$ is constant and it is locally minimizing. The exponential map $\exp _x: T_x M \rightarrow M$ is defined by $\exp _x(v)=\gamma(1)$ for $v \in T_x M$ if there is a geodesic $\gamma:[0,1] \rightarrow M$ with $\gamma(0)=x$ and $\dot{\gamma}(0)=v$. A Finsler manifold $(M, F)$ is said to be forward complete (resp. backward complete) if each geodesic defined on $[0, \ell)$ (resp. $(-\ell, 0])$ can be extended to a geodesic defined on $[0, \infty)$ (resp. $(-\infty, 0])$. We say $(M, F)$ is complete if it is both forward complete and backward complete. By Hopf-Rinow theorem on forward complete Finsler manifolds, any two points in $M$ can be connected by a minimal forward geodesic and the forward closed balls $\overline{B_R^{+}(p)}$ are compact (\cite{BaoChSh,Shen1}).

For a point $p \in M$ and a unit vector $y \in S_{p}M:=\{v \in T_{p}M \mid F(p,v)=1\}$, the cut value $i_y$ of $y$ is defined by
$$
i_{y} :=\sup \left\{t>0 \mid d(p, \exp _{p}(ty))=t \right\}.
$$
Further, we define the injectivity radius $i_p$ at $p$ by $i_{p}:=\inf\limits_{y \in S_{p} M} i_{y}$. Let
$$
\mathcal{D}_{p}:=\left\{\exp _{p}(t y)\mid 0 \leq t< i_{y}, y \in S_{p} M\right\} \subset M
$$
and
$$
Cut_{p}:=M-\mathcal{D}_{p}.
$$
$Cut_{p}$ and $\mathcal{D}_{p}$ are called the cut-locus and the cut-domain of $p$ respectively. Let
$$
\widetilde{\mathcal{D}}_{p}:=\left\{ty \mid \ 0 \leq t< i_{y}, y \in S_{p} M\right\} \subset T_{p}M .
$$
$\widetilde{\mathcal{D}}_{p}$ is called the tangent cut-domain at $p$. The exponential map
$$
\exp _{p}: \widetilde{\mathcal{D}}_{p} \rightarrow \mathcal{D}_{p}
$$
is an onto diffeomorphism. The cut-locus $Cut_{p}$ of $p$ always has null measure and $d_{p}:= d (p, \cdot)$ is $C^1$ outside the cut-locus of $p$ (see \cite{BaoChSh, Shen1}).

Let $(M, F, m)$ be an $n$-dimensional Finsler metric measure manifold. Write the volume form $dm$ of $m$ as $d m = \sigma(x) dx^{1} dx^{2} \cdots d x^{n}$. Define
\be\label{Dis}
\tau (x, y):=\ln \frac{\sqrt{{\rm det}\left(g_{i j}(x, y)\right)}}{\sigma(x)}.
\ee
We call $\tau$ the distortion of $F$. It is natural to study the rate of change of the distortion along geodesics. For a vector $y \in T_{x} M \backslash\{0\}$, let $\gamma=\gamma(t)$ be the geodesic with $\gamma(0)=x$ and $\dot{\gamma}(0)=y.$  Set
\be
{\bf S}(x, y):= \frac{d}{d t}\left[\tau(\gamma(t), \dot{\gamma}(t))\right]|_{t=0}.
\ee
$\mathbf{S}$ is called the S-curvature of $(M, F, m)$ \cite{ChernShen, shen}. We say that $\mathbf{S}\geq \kappa$ for some $\kappa \in \mathbb{R}$ if $\mathbf{S}(v)\geq \kappa F(v)$ for all $v\in TM$.

Let $Y$ be a $C^{\infty}$ geodesic field on an open subset $U \subset M$ and $\hat{g}=g_{Y}.$  Let
\be
d m:=e^{- \psi} {\rm Vol}_{\hat{g}}, \ \ \ {\rm Vol}_{\hat{g}}= \sqrt{{det}\left(g_{i j}\left(x, Y_{x}\right)\right)}dx^{1} \cdots dx^{n}. \label{voldecom}
\ee
It is easy to see that $\psi$ is given by
$$
\psi (x)= \ln \frac{\sqrt{\operatorname{det}\left(g_{i j}\left(x, Y_{x}\right)\right)}}{\sigma(x)}=\tau\left(x, Y_{x}\right),
$$
which is just the distortion of $F$ along $Y_{x}$ at $x\in M$ \cite{ChernShen, Shen1}. Let $y := Y_{x}\in T_{x}M$ (that is, $Y$ is a geodesic extension of $y\in T_{x}M$). Then, by the definitions of the S-curvature, we have
\beqn
&&  {\bf S}(x, y)= Y[\tau(x, Y)]|_{x} = d \psi (y),  \\
&&  \dot{\bf S}(x, y)= Y[{\bf S}(x, Y)]|_{x} =y[Y(\psi)],
\eeqn
where $\dot{\bf S}(x, y):={\bf S}_{|m}(x, y)y^{m}$ and ``$|$" denotes the horizontal covariant derivative with respect to the Chern connection  \cite{shen, Shen1}. Furthermore, the weighted Ricci curvatures are defined as follows \cite{ChSh,OHTA}
\beq
{\rm Ric}_{N}(y)&=& {\rm Ric}(y)+ \dot{\bf S}(x, y) -\frac{{\bf S}(x, y)^{2}}{N-n},   \label{weRicci3}\\
{\rm Ric}_{\infty}(y)&=& {\rm Ric}(y)+ \dot{\bf S}(x, y). \label{weRicciinf}
\eeq
We say that Ric$_N\geq K$ for some $K\in \mathbb{R}$ if Ric$_N(v)\geq KF^2(v)$ for all $v\in TM$, where $N\in \mathbb{R}\setminus \{n\}$ or $N= \infty$.

\begin{rem} On a Riemannian manifold $(M, g)$, the consideration of weighted measure of the form $e^{-f}{\rm Vol}_{g}$ arises naturally in various situations, where $f$ is a smooth function and ${\rm Vol}_{g}$ denotes the standard volume form induced by the metric $g$. On a smooth Riemannian metric measure space $\left(M, g, e^{-f} \mathrm{Vol}_{g}\right)$,  the Bakry-\'{E}mery Ricci tensor is defined as
$$
{\rm Ric}_{f}= {\rm Ric}+{\rm Hess} f.
$$
\end{rem}

\vskip 3mm

For any Finsler metric $F$, its dual metric
\begin{equation}\label{co-Finsler}
F^{*}(x, \xi):=\sup\limits_{y\in T_{x}M\setminus \{0\}} \frac{\xi (y)}{F(x,y)}, \ \ \forall \xi \in T^{*}_{x}M
\end{equation}
is a Finsler co-metric. According to Lemma 3.1.1 in \cite{Shen1}, for any vector $y\in T_{x}M\setminus \{0\}$, $x\in M$, the covector $\xi =g_{y}(y, \cdot)\in T^{*}_{x}M$ satisfies
\be
F(x,y)=F^{*}(x, \xi)=\frac{\xi (y)}{F(x,y)}. \label{shenF311}
\ee
Conversely, for any covector $\xi \in T_{x}^{*}M\setminus \{0\}$, there exists a unique vector $y\in T_{x}M\setminus \{0\}$ such that $\xi =g_{y}(y, \cdot)\in T^{*}_{x}M$ . Naturally,  we define a map ${\cal L}: TM \rightarrow T^{*}M$ by
$$
{\cal L}(y):=\left\{
\begin{array}{ll}
g_{y}(y, \cdot), & y\neq 0, \\
0, & y=0.
\end{array} \right.
$$
It follows from (\ref{shenF311}) that
$$
F(x,y)=F^{*}(x, {\cal L}(y)).
$$
Thus ${\cal L}$ is a norm-preserving transformation. We call ${\cal L}$ the Legendre transformation on Finsler manifold $(M, F)$.

Let
$$
g^{*kl}(x,\xi):=\frac{1}{2}\left[F^{*2}\right]_{\xi _{k}\xi_{l}}(x,\xi).
$$
For any $\xi ={\cal L}(y)$, we have
\be
g^{*kl}(x,\xi)=g^{kl}(x,y), \label{Fdual}
\ee
where $\left(g^{kl}(x,y)\right)= \left(g_{kl}(x,y)\right)^{-1}$. If $F$ is uniformly smooth and convex with (\ref{usk}), then  $\left(g^{*ij}\right)$ is uniformly elliptic in the sense that there exists two constants $\tilde{\kappa}=(\kappa^*)^{-1}$, $\tilde{\kappa}^*=\kappa^{-1}$ such that for $x \in M, \ \xi \in T^*_x M \backslash\{0\}$ and $\eta \in T_x^* M$, we have \cite{OHTA}
\be
\tilde{\kappa}^* F^{* 2}(x, \eta) \leq g^{*i j}(x, \xi) \eta_i \eta_j \leq \tilde{\kappa} F^{* 2}(x, \eta). \label{unisc}
\ee

Given a smooth function $u$ on $M$, the differential $d u_x$ at any point $x \in M$,
$$
d u_x=\frac{\partial u}{\partial x^i}(x) d x^i
$$
is a linear function on $T_x M$. We define the gradient vector $\nabla u(x)$ of $u$ at $x \in M$ by $\nabla u(x):=\mathcal{L}^{-1}(d u(x)) \in T_x M$. In a local coordinate system, we can express $\nabla u$ as
\be \label{nabna}
\nabla u(x)= \begin{cases}g^{* i j}(x, d u) \frac{\partial u}{\partial x^i} \frac{\partial}{\partial x^j}, & x \in M_u, \\ 0, & x \in M \backslash M_u,\end{cases}
\ee
where $M_{u}:=\{x \in M \mid d u(x) \neq 0\}$ \cite{Shen1}. In general, $\nabla u$ is only continuous on $M$, but smooth on $M_{u}$.

The Hessian of $u$ is defined by using Chern connection as
$$
\nabla^2 u(X, Y)=g_{\nabla u}\left(D_X^{\nabla u} \nabla u, Y\right).
$$
One can show that $\nabla^2 u(X, Y)$ is symmetric (see \cite{OS2, WuXin}).

Let
$$
W^{1, 2}(M)=\left\{u \in W^{1,2}_{\rm loc}(M)\cap L^2(M) \mid \int_{M} \left( F^{*2}(du)+ \overleftarrow{F}^{*2}(du) \right) dm <\infty \right\}
$$
and $W_{0}^{1, 2}(M)$ be the closure of $\mathcal{C}_{0}^{\infty}(M)$ under the norm
$$
\|u\|_{W^{1, 2}(M)}:=\|u\|_{L^2(M)}+\frac{1}{2}\|F^{*}(d u)\|_{L^2(M)}+\frac{1}{2}\|\overleftarrow{F}^{*}(d u)\|_{L^2(M)},
$$
where $\mathcal{C}_{0}^{\infty}(M)$ denotes the set of all smooth compactly supported functions on M.

Now, we decompose the volume form $dm$ of $m$ as $d m=\mathrm{e}^{\Phi} d x^1 d x^2 \cdots d x^n$. Then the divergence of a differentiable vector field $V$ on $M$ is defined by
$$
\operatorname{div}_m V:=\frac{\partial V^i}{\partial x^i}+V^i \frac{\partial \Phi}{\partial x^i}, \quad V=V^i \frac{\partial}{\partial x^i} .
$$
One can also define $\operatorname{div}_m V$ in the weak form by following divergence formula
$$
\int_M \phi \operatorname{div}_m V d m=-\int_M d \phi(V) d m
$$
for all $\phi \in \mathcal{C}_0^{\infty}(M)$. Further, the Finsler Laplacian $\Delta u$ is defined by
\be
\Delta u:=\operatorname{div}_m(\nabla u). \label{Lapla}
\ee
From (\ref{Lapla}), Finsler Laplacian is a nonlinear elliptic differential operator of the second order. Moreover, noticing that $\nabla u$ is weakly differentiable, the Finsler Laplacian should be understood in a weak sense, that is, for $u \in W^{1,2}(M)$, $\Delta u$ is defined by
\be
\int_M \phi \Delta u d m:=-\int_M d \phi(\nabla u) dm  \label{Lap1}
\ee
for $\phi \in \mathcal{C}_0^{\infty}(M)$ \cite{OHTA,Shen1}.

Given a weakly differentiable function $u$ and a vector field $V$ which does not vanish on $M_u$, the weighted Laplacian of $u$ on the weighted Riemannian manifold $\left(M, g_V, m\right)$ is defined  by
$$
\Delta^{V} u:= {\rm div}_{m}\left(\nabla^V u\right),
$$
where
$$
\nabla^V u:= \begin{cases}g^{ij}(x, V) \frac{\partial u}{\partial x^i} \frac{\partial}{\partial x^j} & \text { for } x \in M_u, \\ 0 & \text { for } x \notin M_u .\end{cases}
$$
Similarly, the weighted Laplacian can be viewed in a weak sense. We note that $\nabla^{\nabla u}u=\nabla u$ and $\Delta^{\nabla u} u=$ $\Delta u$. Moreover, it is easy to see that $\Delta u= {\rm tr}_{\nabla u} \nabla^2 u-{\bf S}(\nabla u)$ on $M_u$ \cite{OHTA,Shen1,WuXin}.

The following Bochner-Weitzenb\"{o}ck type formula established by Ohta-Sturm \cite{OS2} is very important to derive gradient estimates for harmonic function in this paper.

\begin{thm}{\rm (\cite{OHTA,OS2})}\label{boch}   For $u \in C^{\infty}(M)$, we have
\be
\Delta^{\nabla u}\left[\frac{F^2(\nabla u)}{2}\right]-d(\Delta u)(\nabla u)=\operatorname{Ric}_{\infty}(\nabla u)+\left\|\nabla^2 u\right\|_{\mathrm{HS}(\nabla u)}^2 \label{poinBo}
\ee
on $M_{u}$. Moreover, for $u \in H_{\mathrm{loc}}^{2}(M) \bigcap C^1(M)$ with $\Delta u \in H_{\mathrm{loc}}^{1}(M)$, we have
\beq
-\int_M d \phi\left(\nabla^{\nabla u}\left[\frac{F^2(x, \nabla u)}{2}\right]\right) d m =\int_M \phi \left\{d \left(\Delta u\right)(\nabla u)+{\rm Ric}_{\infty}(\nabla u)+\left\|\nabla^2 u\right\|_{{\rm HS}(\nabla u)}^2\right\} d m \label{BWforinf}
\eeq
for all bounded functions $\phi \in H_{0}^{1}(M) \bigcap L^{\infty}(M)$. Here $\left\|\cdot \right\|_{{\rm HS}(\nabla u)}$ denotes the Hilbert-Schmidt norm with respect to $g_{\nabla u}$.
\end{thm}

\section{Integral weighted Ricci curvature}\label{Integral}

Let $(M, F, m)$ be an $n$-dimensional Finsler manifold with a smooth measure $m$ and $x \in M$. Remember that $\widetilde{\cal D}_{x}=\left\{ty \mid 0\leq t< i_{y}, y\in S_{x}M\right\}$ and $\mathcal{D}_{x}=\exp_{x} (\widetilde{\cal D}_{x})$. For any $z \in \mathcal{D}_x$, we can choose the geodesic polar coordinates $(r, \theta)$ centered at $x$ for $z$ such that $r(z)=F(v)$ and $\theta^{\alpha}(z)=\theta^{\alpha}\left(\frac{v}{F(v)}\right)$ for $1\leq \alpha \leq n-1$, where $r(z)=d(x, z)$  and $v=\exp _x^{-1}(z) \in T_x M \backslash\{0\}$. It is well known that the distance function $r$ starting from $x \in M$ is smooth on $\mathcal{D}_x$ and $F(\nabla r)=1$ \cite{BaoChSh,Shen1}.  A basic fact is that the distance function $r=d(x, \cdot)$ satisfies the following \cite{Shen1,WuXin}
\[
\nabla r |_{z}= \frac{\pa}{\pa r}|_{z}.
\]
By Gauss's lemma, the unit radial coordinate vector $\frac{\partial }{{\partial r}}$ and the coordinate vectors $\frac{\partial }{{\partial {\theta ^\alpha }}}$ for $1\leq \alpha \leq n-1$ are mutually vertical with respect to $g_{\nabla r}$ (\cite{BaoChSh}, Lemma 6.1.1).  Therefore, we can simply write the volume form at $z=\exp _{x}(r\xi)$  with $v=r \xi$ as $dm |_{\exp _{x}(r \xi)}=\sigma(x, r, \theta) dr d\theta$, where $\xi \in S_{x}M$.
Then, for forward geodesic ball $B_{R}=B_R^{+}(x)$ of radius $R$ at the center $x \in M$, the volume of $B_R$ is
\be
m(B_R)=\int_{B_R} d m=\int_{B_R \cap \mathcal{D}_x} d m=\int_0^{R} dr \int_{\mathcal{D}_x(r)} \sigma(x, r, \theta) d\theta=\int_{S_{x}M}\int_0^{\min\{R, i_y\}} \sigma(x, r, \theta) dr d\theta, \label{volBall}
\ee
where $\theta^{\alpha}=\theta^{\alpha}(y)$ for $y\in S_{x}M$ and $\mathcal{D}_x(r)=\left\{y \in S_xM \mid r y \in \widetilde{\cal D}_{x} \right\}$ \cite{OHTA,Shen1}. Obviously, for any $0<s<t<R, \ \mathcal{D}_x(t) \subseteq \mathcal{D}_x(s)$. Besides, by the definition of Laplacian, we have \cite{Shen1, WuXin}
\be
\Delta r=\frac{\partial}{\partial r}\ln \sigma (x, r, \theta). \label{Lapdis}
\ee

Let
$$
{\rm \underline{Ric}_{\infty}}(z):=\min\limits_{y\in T_zM \backslash \{0\}} \frac{{\rm Ric_{\infty}}(z, y)}{F^2(z, y)}
$$
and ${\rm Ric}_{\infty}^{K}(z):= \max \left\{(n-1)K-{\rm \underline{Ric}_{\infty}(z), 0}\right\}$ for some $K \in \mathbb{R}$.  Given $p\geq 1$, $\vartheta\geq0$ and $R>0$. Let
\be
\overline{\|{\rm Ric}_{\infty}^{K}\|}_{p, R, \vartheta}(x):=\left(\frac{1}{m(B_R^{+}(x))}\int_{0}^{R} \int_{\mathcal{D}_x(r)} ({\rm Ric}_{\infty}^K)^p e^{-\vartheta r}\sigma(x, r, \theta) dr d\theta  \right)^{\frac{1}{p}}\label{InweiR1-1}
\ee
and
\be
\overline{\|{\rm Ric}_{\infty}^{K}\|}_{p, R, \vartheta}:=\sup\limits_{x\in M}\overline{\|{\rm Ric}_{\infty}^{K}\|}_{p, R, \vartheta}(x).\label{InweiR1-2}
\ee
$\overline{\|{\rm Ric}_{\infty}^{K}\|}_{p, R, \vartheta}$ measures how much the weighted Ricci curvature ${\rm Ric}_{\infty}$ lies below a given bound $(n-1)K$ in the $L^p$ sense. Clearly, $\overline{\|{\rm Ric}_{\infty}^{K}\|}_{p, R, \vartheta}=0$ iff ${\rm Ric_{\infty}}\geq (n-1)K$. It is often convenient to work with the following scale invariant curvature quantity (with $K=0$)
\be
\overline{\mathcal{K}}_x(p, R, \vartheta)=R^2\left(\frac{1}{m(B_R^{+}(x))}\int_0^{R} \int_{\mathcal{D}_x(r)} \left(-{\rm \underline{Ric}_{\infty}}\right)_+^p e^{-\vartheta r}\sigma(x, r, \theta) dr d\theta  \right)^{\frac{1}{p}}  \label{InweiR2-1}
\ee
and
\be
\overline{\mathcal{K}}(p, R, \vartheta):=\sup\limits_{x\in M} \overline{\mathcal{K}}_x(p, R, \vartheta),\label{InweiR2-2}
\ee
where $( \star )_{+}:=\max\{\star ~, ~ 0\}$.

Noticing that $\Delta r ={\rm tr}_{\nabla r}(\nabla ^{2}r)- {\bf S}(\nabla r)$, let $h:= {\rm tr}_{\nabla r}(\nabla ^{2}r)$ and $H_{K}(r):=(n-1)\frac{s_{K}'(r)}{s_{K}(r)}$, where
$$
{s_K}(t): = \left\{ {\begin{array}{*{20}{c}}
\begin{array}{l}
\frac{1}{{\sqrt K }}\sin (\sqrt K t)\\
t\\
\frac{1}{{\sqrt { - K} }}\sinh (\sqrt { - K} t)
\end{array}&\begin{array}{l}
K > 0,\\
K = 0,\\
K < 0.
\end{array}
\end{array}} \right.
$$
Consider
$$
\varphi(r, \theta): = \left\{\begin{array}{*{20}{c}}
\begin{array}{l}
\left(\Delta r-H_K(r)-\vartheta\right)_{+}\\
0
\end{array}
&\begin{array}{l}
0\leq r< i_x,\\
r\geq i_x.
\end{array}
\end{array} \right.
$$
It is clear that, if $\Delta r\geq H_K+\vartheta$, then $\varphi=\Delta r-H_K(r)-\vartheta$. Otherwise, $\varphi=0$.   Moreover, if $\mathbf{S}(\nabla r)\geq-\vartheta$, then
$$
\lim _{r \rightarrow 0^{+}} \varphi(r, \theta) =\lim _{r \rightarrow 0^{+}}\left(h-{\bf S}(\nabla r)-H_{K}(r)-\vartheta \right)_{+}=\lim _{r \rightarrow 0^{+}}(-{\bf S}(\nabla r)-\vartheta)_{+}=0
$$
because $h \sim \frac{n-1}{r}$ and $H_{K}(r) \sim \frac{n-1}{r}$ when $r \rightarrow 0^{+}$.

On the other hand, by (5.1) in \cite{WuXin}, that is,
$$
\frac{d}{d r} tr_{\nabla r}(\nabla^{2} r)+\sum_{i, j}\left(\nabla^{2}r(E_{i}, E_{j})\right)^{2}= -{\rm Ric}(\nabla r),
$$
where $E_{1}, \cdots, E_{n-1}, E_{n}=\nabla r$ is local $g_{\nabla r}$-orthonormal frame along the geodesic, we have
$$
\frac{\partial}{\partial r}h + \frac{h^2}{n-1} \leq \frac{d}{d r} tr_{\nabla r}(\nabla^{2} r)+\sum_{i, j}\left(\nabla^{2}r(E_{i}, E_{j})\right)^{2}= -{\rm Ric}(\nabla r).
$$
Then
$$
\frac{\partial}{\partial r} \Delta r =\frac{\partial}{\partial r}h - {\rm \dot{\mathbf{S}}}(\nabla r) \leq  - \frac{h^2}{n-1} -{\rm Ric}_{\infty}(\nabla r).
$$
Noticing that $H_K'(r)+\frac{H_K^2(r)}{n-1}=-(n-1)K$, we have
$$
\frac{\partial}{\partial r} \varphi + \frac{h^2}{n-1} -  \frac{H_K^2}{n-1} \leq (n-1)K-{\rm Ric}_{\infty}(\nabla r).
$$

If $\mathbf{S}(\nabla r)\geq-\vartheta$, then $h^2-H_{K}^{2}=\left(\varphi+H_{K}+{\bf S}(\nabla r)+\vartheta\right)^{2}-H_{K}^{2} \geq \varphi^{2}+2\varphi H_{K}$ (assume $r\leq \frac{\pi}{2\sqrt{K}}$ when $K>0$). Then the above inequality implies
\begin{equation}{\label{phi-Ric}}
\frac{\partial}{\partial r} \varphi + \frac{\varphi^2}{n-1} +  2\frac{\varphi H_K}{n-1} \leq (n-1)K-{\rm Ric}_{\infty}(\nabla r)\leq {\rm Ric}^K_{\infty}.
\end{equation}

Now, following the lines of \cite{PPS,PeterW1} or \cite{ZhaoW},  we have the following Laplacian comparison theorem and volume comparison theorems. Firstly, we give the following Laplacian comparison theorem.

\begin{thm}{\label{phi-laplacian}}
Let $(M, F, m)$ be an $n$-dimensional Finsler metric measure manifold. Assume that $\mathbf{S}(\nabla r)\geq-\vartheta$ along the minimal geodesic segment starting from $x_{0}\in M$ for some $\vartheta\geq0$ . Then, for $p>\frac{n}{2}$, $K\in \mathbb{R}$ and $r>0$ ($r\leq \frac{\pi}{2\sqrt{K}}$ when $K>0$), we have
\begin{equation}{\label{laplacian-1}}
\int_{0}^r\varphi^{2p}(t, \theta) e^{-\vartheta t} \sigma(t, \theta) dt\leq \left(\frac{(n-1)(2p-1)}{2p-n}\right)^{p} \int_{0}^{r}\left({\rm Ric}^{K}_{\infty}\right)^p e^{-\vartheta t} \sigma (t, \theta) dt
\end{equation}
and
\begin{equation}{\label{laplacian-2}}
\varphi(r,\theta)^{2p-1}e^{-\vartheta r}\sigma(r, \theta) \leq (2p-1)^p \left( \frac{n-1}{2p-n} \right)^{p-1} \int_0^r\left({\rm Ric}^K_{\infty}\right)^p e^{-\vartheta t}\sigma(t, \theta) dt,
\end{equation}
where $(t, \theta)$ denotes the geodesic polar coordinates centered at $x_{0}$.
\end{thm}

\begin{proof}
Multiplying inequality (\ref{phi-Ric}) by $(2p-1)\varphi(t, \theta)^{2p-2}\sigma(t, \theta)$ yields
$$
\left(\varphi^{2p-1}\right)'\sigma+\frac{2p-1}{n-1}\varphi^{2p}\sigma+\frac{2(2p-1)}{n-1}\varphi^{2p-1}H_K\sigma \leq (2p-1){\rm Ric}^K_{\infty} \varphi^{2p-2}\sigma,
$$
where $\varphi'(t, \theta):=\frac{\pa}{\pa t}\varphi(t, \theta)$. Since
$$
\left(\varphi^{2p-1}\right)'\sigma=\left(\varphi^{2p-1}\sigma\right)'-\varphi^{2p-1}\sigma \Delta r \geq \left(\varphi^{2p-1}\sigma\right)'-\varphi^{2p-1}\sigma \left(\varphi+H_K+\vartheta\right),
$$
then the above inequality can be rewritten as
$$
\left(\varphi^{2p-1}\sigma\right)'+ \frac{2p-n}{n-1}\varphi^{2p}\sigma + \frac{4p-n-1}{n-1}\varphi^{2p-1}\sigma H_K - \varphi^{2p-1}\sigma \vartheta \leq (2p-1){\rm Ric}^K_{\infty} \varphi^{2p-2}\sigma.
$$
Multiplying the above inequality by $e^{-\vartheta t}$ yields
$$
\left(\varphi^{2p-1}e^{-\vartheta t}\sigma\right)'+ \frac{2p-n}{n-1}\varphi^{2p}e^{-\vartheta t}\sigma + \frac{4p-n-1}{n-1}\varphi^{2p-1}e^{-\vartheta t}\sigma H_K \leq (2p-1)e^{-\vartheta t}{\rm Ric}^K_{\infty} \varphi^{2p-2}\sigma.
$$
By the assumptions, $p>\frac{n}{2}$, $r\leq \frac{\pi}{2\sqrt{K}}$ when $K>0$ and $\varphi(0)=0$, we can throw away the third term in left hand side of above inequality, and then, integrate in $t$ from $0$ to $r$ to get
$$
\varphi(r, \theta)^{2p-1}e^{-\vartheta r}\sigma(r, \theta)+ \frac{2p-n}{n-1}\int_{0}^{r}\varphi^{2p}e^{-\vartheta t}\sigma dt\leq (2p-1)\int_{0}^{r}{\rm Ric}^K_{\infty} \varphi^{2p-2}e^{-\vartheta t}\sigma dt,
$$
which implies
\begin{equation}{\label{laplacian-1-1}}
\varphi(r, \theta)^{2p-1}e^{-\vartheta r}\sigma(r, \theta) \leq (2p-1)\int_{0}^{r}{\rm Ric}^K_{\infty} \varphi^{2p-2}e^{-\vartheta t}\sigma dt
\end{equation}
and
\begin{equation}{\label{laplacian-1-2}}
\int_{0}^{r}\varphi^{2p}e^{-\vartheta t}\sigma dt\leq \frac{n-1}{2p-n}(2p-1)\int_{0}^{r}{\rm Ric}^K_{\infty} \varphi^{2p-2}e^{-\vartheta t}\sigma dt.
\end{equation}
By H\"{o}lder's inequality, we have
\begin{equation}{\label{laplacian-1-3}}
\int_{0}^{r}{\rm Ric}^K_{\infty} \varphi^{2p-2}e^{-\vartheta t}\sigma dt \leq \left(\int_0^r\varphi^{2p}e^{-\vartheta t}\sigma dt\right)^{1-\frac{1}{p}}
\left(\int_0^r\left({\rm Ric}^K_{\infty}\right)^p e^{-\vartheta t}\sigma dt\right)^{\frac{1}{p}}.
\end{equation}
Combining (\ref{laplacian-1-2}) with (\ref{laplacian-1-3}), we immediately get (\ref{laplacian-1}). Then applying (\ref{laplacian-1}) and (\ref{laplacian-1-3}) to (\ref{laplacian-1-1}), we obtain (\ref{laplacian-2}).
\end{proof}

\vskip 2mm

From (\ref{laplacian-1}), it is easy to see that

\beq
\int_{B_{r}^{+}(x_0)} \varphi^{2 p} d m & \leq & e^{\vartheta r} \left(\frac{(n-1)(2 p-1)}{2 p-n}\right)^{p} \int_{0}^{r} \int_{{\cal D}_{x_{0}}(t)}\left({\rm Ric}_{\infty}^{K}\right)^{p} e^{-\vartheta t} \sigma(t, \theta) d \theta dt \nonumber\\
& =& e^{\vartheta r} \left(\frac{(n-1)(2 p-1)}{2 p-n}\right)^{p}\left\|{\rm Ric}_{\infty}^{K}\right\|_{p, r, \vartheta}^{p}(x_0), \label{laplacian-2}
\eeq
where $\|{\rm Ric}_{\infty}^{K}\|_{p, r, \vartheta}(x_0) = m\left(B_{r}^{+}(x_0)\right)^{\frac{1}{p}} \overline{\left\|{\rm Ric}_{\infty}^{K}\right\|}_{p, r, \vartheta}(x_0)$.

In the following, let
\be
v(n,K,r,\vartheta):={\rm Vol}(\mathbb{S}^{n-1})\int_0^{r}e^{\vartheta t}s_K^{n-1}(t)dt,  \label{vnK}
\ee
where ${\rm Vol}(\mathbb{S}^{n-1})$ denotes Euclidean volume of unit sphere $\mathbb{S}^{n-1}$ in $\mathbb{R}^n$. From Theorem \ref{phi-laplacian}, we can get the following volume comparison.

\begin{thm}{\label{phi-volume}}
Let $(M, F, m)$ be an n-dimensional Finsler metric measure manifold. Assume that $\mathbf{S}(\nabla r)\geq-\vartheta$ along the minimal geodesic segment starting from $x_{0}\in M$ for some $\vartheta\geq 0$ . Then, for $p>\frac{n}{2}$, $K\in \mathbb{R}$ and $0<r\leq R$ ($R\leq \frac{\pi}{2\sqrt{K}}$ when $K>0$), there exists a constant  $C(n,p, \vartheta, K, R)$ such that
\begin{equation}{\label{volume-1}}
\left(\frac{m(B^+_R(x_0))}{v(n,K,R,\vartheta)}\right)^{\frac{1}{2p}}-\left(\frac{m(B^+_r(x_0))}{v(n,K,r,\vartheta)}\right)^{\frac{1}{2p}}
\leq C(n,p, \vartheta, K, R) \overline{\|{\rm Ric}_{\infty}^K\|}_{p, R, \vartheta}^{\frac{1}{2}}(x_0),
\end{equation}
where
\beqn
& & C(n,p, \vartheta, K, R) \\
& &:=\frac{1}{2p}\left(\frac{(n-1)(2p-1)}{2p-n}\right)^{\frac{1}{2}} m(B_{R}^{+}(x_0))^{\frac{1}{2p}}{\rm Vol(\mathbb{S}^{n-1})} \int_{0}^{R} t e^{(1+\frac{1}{2p})\vartheta t}s_{K}^{n-1}(t) v(n,K,t,\vartheta)^{-\frac{2p+1}{2p}} dt
\eeqn
is nondecreasing in $R$.
\end{thm}

\begin{proof}
Recall $\frac{\partial}{\partial r} \sigma (r, \theta)=\sigma(r, \theta)\Delta r$ and $\frac{\partial}{\partial r} \left(e^{\vartheta r} s_K^{n-1}(r)\right)=e^{\vartheta r} s_K^{n-1}(r) (H_K(r)+\vartheta)$. We have
\begin{equation}{\label{volume-2}}
\frac{\partial}{\partial r} \left(\frac{\sigma (r, \theta)}{e^{\vartheta r}s_K^{n-1}(r)}\right)=\left(\Delta r-H_K(r)-\vartheta\right)\frac{\sigma(r, \theta)}{e^{\vartheta r}s_K^{n-1}(r)}\leq \varphi \frac{\sigma(r, \theta)}{e^{\vartheta r}s_K^{n-1}(r)}.
\end{equation}
Integrating the both sides of (\ref{volume-2}) from $t$ to $r$ yields
\[
\sigma(r, \theta) e^{\vartheta t} s_{K}^{n-1}(t)-\sigma(t, \theta) e^{\vartheta r} s_{K}^{n-1}(r) \leq e^{\vartheta r} s_{K}^{n-1}(r) \int_{0}^{r} \varphi(s, \theta) \sigma(s, \theta) ds.
\]
Since $\mathcal{D}_{x_0}(r)\subset \mathcal{D}_{x_0}(t)$ for any $t<r$, we have
$$
\int_{\mathcal{D}_{x_0}(r)}\sigma (r, \theta) e^{\vartheta t}s_{K}^{n-1}(t) d\theta -  \int_{\mathcal{D}_{x_0}(t)}\sigma (t, \theta) e^{\vartheta r}s_{K}^{n-1}(r) d\theta
\leq  e^{\vartheta r}s_{K}^{n-1}(r) \int_{0}^{r} \int_{\mathcal{D}_{x_0}(s)}\varphi (s, \theta) \sigma(s, \theta) ds d\theta.
$$
Thus, by (\ref{volBall}), we have
\beqn
\frac{d}{dr}\left(\frac{m\left(B^{+}_{r}(x_0)\right)}{v(n,K,r,\vartheta)}\right)&=&\frac{{\rm Vol}(\mathbb{S}^{n-1})\int_{0}^{r}\left[\int_{\mathcal{D}_{x_0}(r)}\sigma(r, \theta)e^{\vartheta t}s_{K}^{n-1}(t)d\theta -\int_{\mathcal{D}_{x_0}(t)} \sigma(t, \theta) e^{\vartheta r}s_{K}^{n-1}(r)d\theta\right]dt}{v(n,K,r,\vartheta)^2}\\
&\leq & \frac{r e^{\vartheta r}s_{K}^{n-1}(r){\rm Vol}(\mathbb{S}^{n-1})\int_{B_{r}^{+}(x_0)} \varphi \ dm}{v(n,K,r,\vartheta)^2}\\
&\leq & \frac{r e^{\vartheta r}s_{K}^{n-1}(r){\rm Vol}(\mathbb{S}^{n-1})\left(\int_{B_{r}^{+}(x_0)} \varphi^{2p} dm\right)^{\frac{1}{2p}} m(B_{r}^{+}(x_0))^{1-\frac{1}{2p}}}{v(n,K,r,\vartheta)^2}\\
&\leq & C(n,p)\frac{r e^{(\frac{1}{2p}+1)\vartheta r}s_{K}^{n-1}(r){\rm Vol}(\mathbb{S}^{n-1}) \|{\rm Ric}_{\infty}^K\|_{p, r, \vartheta}^{\frac{1}{2}}(x_0) m(B_{r}^+(x_0))^{1-\frac{1}{2p}}}{v(n,K,r,\vartheta)^2}\\
&=& C(n,p) r s_{K}^{n-1}(r)e^{(\frac{1}{2p}+1)\vartheta r}{\rm Vol}(\mathbb{S}^{n-1})\|{\rm Ric}_{\infty}^{K}\|_{p, r, \vartheta}^{\frac{1}{2}}(x_0)\left(\frac{m\left(B^+_r(x_0)\right)}{v(n,K,r,\vartheta)}\right)^{1-\frac{1}{2p}} \\
& & \times\left(\frac{1}{v(n,K,r,\vartheta)}\right)^{1+\frac{1}{2p}},
\eeqn
where $C(n, p)=\left( \frac{(n-1)(2p-1)}{2p-n} \right)^{\frac{1}{2}}$ and we have used (\ref{laplacian-2}) in the fourth line. Further, the above inequality can be rewritten as
\be
\frac{d}{dr}\left(\frac{m\left(B^{+}_{r}(x_0)\right)}{v(n,K,r,\vartheta)}\right)^{\frac{1}{2p}}\leq \frac{1}{2p} C(n,p) r s_{K}^{n-1}(r)e^{(1+\frac{1}{2p})\vartheta r}{\rm Vol}(\mathbb{S}^{n-1})\|{\rm Ric}_{\infty}^{K}\|_{p, r, \vartheta}^{\frac{1}{2}}(x_0)\left(\frac{1}{v(n,K,r,\vartheta)}\right)^{1+\frac{1}{2p}}.\label{vocom}
\ee
Finally, integrating the both sides of (\ref{vocom}) from $r$ to $R$, we have
\beqn
& & \left(\frac{m(B^+_R(x_0))}{v(n,K,R,\vartheta)}\right)^{\frac{1}{2p}}-\left(\frac{m(B^+_r(x_0))}{v(n,K,r,\vartheta)}\right)^{\frac{1}{2p}}\\
& &\leq \frac{1}{2p}C(n,p) {\rm Vol}(\mathbb{S}^{n-1})\|{\rm Ric}_{\infty}^K\|_{p, R, \vartheta}^{\frac{1}{2}}(x_0) \int_0^R te^{(1+\frac{1}{2p})\vartheta t}s_K^{n-1}(t)\left(\frac{1}{v(n,K,t,\vartheta)}\right)^{1+\frac{1}{2p}} dt\\
& &:=C(n,p, \vartheta, K, R) \overline{\|{\rm Ric}_{\infty}^K\|}_{p, R, \vartheta}^{\frac{1}{2}}(x_0).
\eeqn
This completes the proof of Theorem \ref{phi-volume}.
\end{proof}

\vskip 2mm

Further, from Theorem \ref{phi-volume},  we can obtain the following Bishop-Gromov type volume comparison.

\begin{thm}{\label{doubling}}
Let $(M, F, m)$ be an n-dimensional Finsler metric measure manifold. Assume that $\mathbf{S}(\nabla r)\geq-\vartheta$ along the minimal geodesic segment starting from $x_{0} \in M$ for some $\vartheta\geq 0$ . Given $p>\frac{n}{2}$, $K\in \mathbb{R}$, $\Xi>1$ and $0<r_1\leq r_2\leq R$ ($R\leq \frac{\pi}{2\sqrt{K}}$ when $K>0$). Then there exists a constant $\varepsilon=\varepsilon(n,p, \vartheta, K, R, \Xi)$,  such that
if $\overline{\|{\rm Ric}_{\infty}^{K}\|}_{p, R, \vartheta}<\varepsilon$, the following inequality holds
\begin{equation}{\label{doubling-1}}
\frac{m(B^+_{r_2}(x_0))}{m(B^+_{r_1}(x_0))}\leq \Xi \ \frac{v(n,K,r_2,\vartheta)}{v(n,K,r_1,\vartheta)}\leq \Xi \ e^{\left(\vartheta+(n-1)\sqrt{|K|}\right)r_2}\left(\frac{r_2}{r_1}\right)^n.
\end{equation}
\end{thm}

\begin{proof}
It follows from Theorem \ref{phi-volume} that
$$
\left(\frac{v(n,K,r_1,\vartheta)}{v(n,K,r_2,\vartheta)}\right)^{\frac{1}{2p}}-\left(\frac{m(B^+_{r_1}(x_0))}{m(B^+_{r_2}(x_0))}\right)^{\frac{1}{2p}}
\leq C(n,p, \vartheta, K, r_2) \overline{\|{\rm Ric}_{\infty}^K\|}_{p, r_2, \vartheta}^{\frac{1}{2}}(x_0) \left(\frac{v(n,K,r_1,\vartheta)}{m(B^+_{r_2}(x_0))} \right)^{\frac{1}{2p}}.
$$
The inequality can be rewritten as
$$
\frac{m(B^+_{r_1}(x_0))}{m(B^+_{r_2}(x_0))}\geq (1-c)^{2p}\frac{v(n,K,r_1,\vartheta)}{v(n,K,r_2,\vartheta)},
$$
where $c:=C(n,p, \vartheta, K, r_2) \overline{\|{\rm Ric}_{\infty}^K\|}_{p, r_2, \vartheta}^{\frac{1}{2}}(x_0) \left(\frac{v(n,K,r_2,\vartheta)}{m(B^+_{r_2}(x_0))} \right)^{\frac{1}{2p}}$. In order to estimate $c$, we use Theorem \ref{phi-volume} again and obtain
\beqn
\left(\frac{v(n,K,r_2,\vartheta)}{m(B^+_{r_2}(x_0))}\right)^{\frac{1}{2p}}
&\leq& \left(\left(\frac{m(B^+_{R}(x_0))}{v(n,K,R,\vartheta)}\right)^{\frac{1}{2p}}-C(n,p, \vartheta, K, R) \overline{\|{\rm Ric}_{\infty}^K\|}_{p, R, \vartheta}^{\frac{1}{2}}\right)^{-1}\\
&=& \left(\frac{v(n,K,R,\vartheta)}{m(B^+_{R}(x_0))}\right)^{\frac{1}{2p}} \left(1 - C(n,p, \vartheta, K, R) \overline{\|{\rm Ric}_{\infty}^{K}\|}_{p, R, \vartheta}^{\frac{1}{2}} \left(\frac{v(n,K,R,\vartheta)}{m(B^+_R(x_0))}\right)^{\frac{1}{2p}}\right)^{-1}.
\eeqn
Then, there exists an $\varepsilon_1=\varepsilon_1(n,p, \vartheta, K, R)$ such that if $\overline{\|{\rm Ric}_{\infty}^{K}\|}_{p, R, \vartheta}<\varepsilon_1$, we have
\begin{equation}\label{v-1}
\left(\frac{v(n,K,r_{2},\vartheta)}{m(B^+_{r_2}(x_0))}\right)^{\frac{1}{2p}}
\leq  2\left(\frac{v(n,K,R,\vartheta)}{m(B^+_{R}(x_0))}\right)^{\frac{1}{2p}}.
\end{equation}
Thus, by the definition of $c$,
\beqn
c&\leq& 2 C(n,p, \vartheta, K, r_2) \overline{\|{\rm Ric}_{\infty}^K\|}_{p, r_2, \vartheta}^{\frac{1}{2}}(x_0) \left(\frac{v(n,K,R,\vartheta)}{m(B^+_{R}(x_0))} \right)^{\frac{1}{2p}} \\
&\leq& 2 C(n,p, \vartheta, K, R) \overline{\|{\rm Ric}_{\infty}^K\|}_{p, R, \vartheta}^{\frac{1}{2}}(x_0) \left(\frac{v(n,K,R,\vartheta)}{m(B^+_{R}(x_0))} \right)^{\frac{1}{2p}}.
\eeqn
Then, for $\Xi >1$, we can choose a constant $\varepsilon_{2} \leq \left( 1-\Xi^{-\frac{1}{2p}}\right)\left(2 C(n,p, \vartheta, K, R)\right)^{-1}\left(\frac{v(n,K,R,\vartheta)}{m(B^+_{R}(x_0))} \right)^{-\frac{1}{2p}}$,  such that if $\overline{\|{\rm Ric}_{\infty}^{K}\|}_{p, R, \vartheta}<\varepsilon_{2}$, we have
$$
c \leq 1-\Xi^{-\frac{1}{2p}}.
$$
Now, choose $\varepsilon:=\min\{\varepsilon_1, \varepsilon_2\}$. Then, when $\overline{\|{\rm Ric}_{\infty}^{K}\|}_{p, R, \vartheta}<\varepsilon$,  we obtain the first inequality of (\ref{doubling-1}). The second inequality is obvious.
\end{proof}

\vskip 2mm

Similar to the proof of Theorem \ref{phi-volume}, we can also easily get the following relative volume comparison.

\begin{thm}{\label{relative-volume}}
Let $(M, F, m)$ be an n-dimensional Finsler metric measure manifold. Assume that $\mathbf{S}(\nabla r)\geq-\vartheta$ along the minimal geodesic segment starting from $x_{0}\in M$  for some $\vartheta \geq 0$.
Then, for $p>\frac{n}{2}$, $K\in \mathbb{R}$ and $0 \leq r_{1}\leq r_{2} < R_{1}\leq R_{2}$ ($R_{2}\leq \frac{\pi}{2\sqrt{K}}$ when $K>0$), the following inequality holds
\begin{equation}{\label{relative-volume-1}}
\left(\frac{m(B^{+}_{r_2, R_2}(x_0))}{v(n,K, r_2, R_2, \vartheta)}\right)^{\frac{1}{2p}} - \left(\frac{m(B^{+}_{r_1, R_1}(x_0))}{v(n,K,r_1,R_1, \vartheta)}\right)^{\frac{1}{2p}}
\leq \widetilde{C} \overline{\|{\rm Ric}_{\infty}^{K}\|}_{p, R_2, \vartheta}^{\frac{1}{2}},
\end{equation}
where
\[
B^{+}_{r, R}(x_0)):= B^{+}_{R}(x_{0}) \backslash B^{+}_{r}(x_{0}), \ \ v(n,K, r, R, \vartheta):={\rm Vol}(\mathbb{S}^{n-1})\int_r^R e^{\vartheta t}s_K^{n-1}(t)dt
\]
and
\beqn
\widetilde{C}& :=& \frac{1}{2p}\left(\frac{(n-1)(2p-1)}{2p-n}\right)^{\frac{1}{2}}{\rm Vol(\mathbb{S}^{n-1})}m(B^{+}_{R_2}(x_0))^{\frac{1}{2p}}\\
& &  \times \left[\left(R_1 e^{(1+\frac{1}{2p})\vartheta R_1}\right) s_{K}^{n-1}(R_1) \int_{r_1}^{r_2} v(n,K,t, R_1,\vartheta)^{-\frac{2p+1}{2p}} dt \right. \\
&& \left. +\int_{R_1}^{R_2}s_{K}^{n-1}(t) t e^{(1+\frac{1}{2p})\vartheta t}v(n,K,r_2,t,\vartheta)^{-\frac{2p+1}{2p}}dt\right].
\eeqn
\end{thm}

From  Theorem \ref{relative-volume}, we can obtain the following volume growth estimate.
\begin{thm}{\label{growth}}
Let $(M, F, m)$ be an n-dimensional forward complete non-compact Finsler metric measure manifold. Assume that $\mathbf{S}(\nabla r)\geq 0$ along the minimal geodesic segment starting from any $x_0\in M$. Given any $p>\frac{n}{2}$ and $R\geq1$, there is an $\varepsilon=\varepsilon(n,p, R)$  such that if $\overline{\mathcal{K}}(p, R+1, 0)<\varepsilon$, then $M$ has infinite volume. Moreover, $M$ has at least linear volume growth if $\Lambda_F<\infty$.
\end{thm}
\begin{proof}
Since $K=0$ and $\vartheta=0$, (\ref{relative-volume-1})  can be simplified as
\be
\left(\frac{m(B^{+}_{r_2, R_2}(x_0))}{R_{2}^{n}-r_{2}^{n}}\right)^{\frac{1}{2p}} - \left(\frac{m(B^{+}_{r_1, R_1}(x_0))}{R_{1}^{n}-r_{1}^{n}}\right)^{\frac{1}{2p}}
\leq \tilde{C}_{1} R_{2}^{-1} \overline{\mathcal{K}}^{\frac{1}{2}}(p, R_2, 0)  \label{volgrow}
\ee
for $0\leq r_1 \leq r_2 < R_1 \leq R_2$, where
\beqn
\tilde{C}_{1} &=&\frac{n}{2p}\left(\frac{(n-1)(2p-1)}{2p-n}\right)^{\frac{1}{2}}\left[ R_1^{n} \int_{r_1}^{r_2} (R_{1}^{n}-t^n)^{-\frac{2p+1}{2p}} dt+\int_{R_1}^{R_2} t^{n} \left(t^{n}-r_{2}^n\right)^{-\frac{2p+1}{2p}}dt\right]m(B^{+}_{R_2}(x_0))^{\frac{1}{2p}}\\
&\leq & \frac{n}{p}\left(\frac{(n-1)(2p-1)}{2p-n}\right)^{\frac{1}{2}} R_{2}^{n+1} \left(\frac{1}{R_{1}^{n}-r_{2}^n}\right)^{\frac{2p+1}{2p}}m(B^{+}_{R_2}(x_0))^{\frac{1}{2p}}\\
&:= & C(n, p) R_{2}^{n+1} \left(\frac{1}{R_{1}^{n}-r_{2}^n}\right)^{\frac{2p+1}{2p}}m(B^{+}_{R_2}(x_0))^{\frac{1}{2p}}.
\eeqn

Letting $r_1=0$, $r_2=R-1$, $R_1=R$ and $R_2=R+1$ in (\ref{volgrow}), then we have
\beqn
\frac{m(B^+_{R+1}(x_0))-m(B^+_{R-1}(x_0))}{(R+1)^n-(R-1)^n}
&\leq & \left[\left(\frac{m(B^+_{ R}(x_0))}{R^n}\right)^{\frac{1}{2p}}
   +  C(n, p)\overline{\mathcal{K}}^{\frac{1}{2}}(p, R+1, 0) (R+1)^{n}m(B^+_{R+1}(x_0))^{\frac{1}{2p}}\right]^{2p}\\
&\leq &C_1(n, p) \frac{m(B^+_{ R}(x_0))}{R^n} + C_2(n, p) (R+1)^{2pn} \overline{\mathcal{K}}^p(p, R+1, 0) m(B^+_{R+1}(x_0)),
\eeqn
where we have used $(a+b)^{2p}\leq 2^{2p-1}(a^{2p}+b^{2p})$ for $a\geq0$ and $b\geq0$. Multiplying the both sides of this inequality by $\frac{(R+1)^n-(R-1)^n}{m(B^+_{R+1}(x_0))}$ yields
$$
\frac{m(B^{+}_{R+1}(x_0))-m(B^{+}_{R-1}(x_0))}{m(B^{+}_{R+1}(x_0))} \leq  \frac{C_3(n, p)}{R} + C_4(n, p) (R+1)^{(2p+1)n} \overline{\mathcal{K}}^p(p, R+1, 0).
$$
Now we choose $\varepsilon=\varepsilon(n, p, R)\leq \left(R (R+1)^{(2p+1)n}\right)^{-\frac{1}{p}}$ small enough. Obviously, if $\overline{\mathcal{K}}(p, R+1, 0)<\varepsilon$, then
\be
\frac{m(B^+_{R+1}(x_0))-m(B^+_{R-1}(x_0))}{m(B^+_{R+1}(x_0))}\leq \frac{C_5(n,p)}{R}. \label{volgroe1}
\ee

Let $x$ be a point with $d(x_{0}, x)=R \geq 1$. For any $z \in B_{1}^{+}(x) \cap B_{1}^{-}(x)$,  $d(x, z)<1, \ d(z, x)<1$ and $d(x_0, z) \leq d(x_0, x)+d(x, z)< R+1$. On the other hand,
$d(x_0, z)  \geq d(x_0, x)-d(z, x) >R-1$.  Then $B^{+}_{1}(x)\cap B^{-}_{1}(x)\subset B^{+}_{R+1}(x_0)\backslash B^{+}_{R-1}(x_0)$ and
\[
m\left(B_{1}^{+}(x) \cap B_{1}^{-}(x)\right) \leq m\left(B_{R+1}^{+}(x_0)\right)-m\left(B_{R-1}^{+}(x_0)\right).
\]
From (\ref{volgroe1}), we have
$$
m(B^{+}_{R+1}(x_0))\geq C_{6}(n, p)m\left(B^{+}_{1}(x)\cap B^{-}_{1}(x)\right) R,
$$
which inplies that $M$ has infinite volume.

If the reversibility $\Lambda_F<\infty$, then we have $B^{+}_{R+1}(x_0) \subset B^{+}_{(\Lambda_F+2)R}(x)$. Thus, we obtain
$$
m(B^{+}_{(\Lambda_{F} +2)R}(x))\geq C_{6}(n, p) m\left(B^{+}_{1}(x)\cap B^{-}_{1}(x)\right) R.
$$
This finishes the proof.
\end{proof}

From Theorem \ref{doubling} and following the arguments in \cite{KristalyZ, ShenZ}, we can also easily derive the following  Gromov type pre-compactness in the case with  integral weighted Ricci curvature bounds.

\begin{thm}\label{pre-com1}
Given any  $p> \frac{n}{2}$,  $\vartheta\geq 0$, $\theta \geq 1$, $K\in \mathbb{R}$ and $D$, $V\in(0, \infty)$  ($D\leq \frac{\pi}{2\sqrt{K}}$ when $K>0$), then there exists a constant $\varepsilon=\varepsilon(n,p, \vartheta, K, D)$ such that the collection of closed Finsler metric measure manifolds $(M, F, m)$ with
$$
\overline{\|{\rm Ric}_{\infty}^{K}\|}_{p, D, \vartheta}<\varepsilon, \ \mathbf{S}\geq -\vartheta, \ \Lambda_{F}\leq \theta, \ {\rm Diam}(M)\leq D, \ m(M)\leq V
$$
is pre-compact in the measured $\theta$-Gromov-Hausdorff topology.
\end{thm}

\begin{thm}\label{pre-com2}
Given any $p> \frac{n}{2}$, $V\in(0, \infty)$, $\vartheta\geq 0$, $K\in \mathbb{R}$, $r>0$ ($r\leq \frac{\pi}{2\sqrt{K}}$ when $K>0$) and letting $\Theta: [0, \infty)\rightarrow [1, \infty)$ be a nondecreasing function, then there exists a constant $\varepsilon=\varepsilon(n,p, \vartheta, K, r)$ such that the collection of pointed forward complete Finsler metric measure manifolds ($M, \star, F, m$) with
$$
\overline{\|{\rm Ric}_{\infty}^{K}\|}_{p, r, \vartheta}<\varepsilon, \ \mathbf{S}\geq -\vartheta, \ \Lambda_F(B^+_r(\star))\leq \Theta(r), \  m(B^+_r(\star))\leq V
$$
is pre-compact in the pointed measured forward $\Theta$-Gromov-Hausdorff topology.
\end{thm}

\section{Dirichlet isoperimetric constant estimate}{\label{Sobolev}}

Let $(M, F, m)$ be an n-dimensional Finsler metric measure manifold and $\Omega \subseteq M$ be a domain with smooth  boundary $\partial \Omega \neq \emptyset$. For any $x \in \partial \Omega$, there exist exactly two unit norm vectors $\mathbf{n}$ which are characterized by
$$
T_{x}(\partial \Omega)=\left\{V \in T_{x} M \mid g_{\mathbf{n}}(\mathbf{n}, V)=0, \ g_{\mathbf{n}}(\mathbf{n}, \mathbf{n})=1\right\} .
$$
Note that, if $\mathbf{n}$ is a norm vector on $\pa \Omega$,  $- \mathbf{n}$ may not be a norm vector unless $F$ is reversible. Denote by $\mathbf{n}_{+}$ (resp. $\mathbf{n}_{-}$) the normal vector that points outward (resp. inward) $\Omega$. The  normal vector  $\mathbf{n}_{+}$ or $\mathbf{n}_{-}$  induces a volume form $\mathrm{d} \nu_{+}$ or $\mathrm{d} \nu_{-}$ on $\pa \Omega$ from $dm$ respectively (see Section 2.4, \cite{Shen1}).
The volumes of $\partial \Omega$ with respect to ${\rm d} \nu_{+}$ and ${\rm d}{\nu}_{-}$ are written respectively by
$$
\nu_{+}(\partial \Omega)=\int_{\partial \Omega} \mathrm{d} \nu_{+} \quad \text { and } \quad \nu_{-}(\partial \Omega)=\int_{\partial \Omega} \mathrm{~d} \nu_{-}.
$$
Note that $\nu_{+}(\partial \Omega) \neq \nu_-(\partial \Omega)$ in general. In fact, the ratio $\frac{\nu_{+}(\partial \Omega)}{\nu_{-}(\partial \Omega)}$ or $\frac{\nu_{-}(\partial \Omega)}{\nu_{+}(\partial \Omega)}$ may be very large.

Now we give the definitions of Dirichlet isoperimetric constant and Sobolev constant on Finsler manifolds and prove an important relation between them.

\begin{Def}
Let $(M, F, m)$ be a forward complete non-compact Finsler metric measure manifold  or a compact Finsler metric measure manifold  with $\partial M\neq \emptyset$. For $n\leq \alpha \leq \infty$, the Dirichlet $\alpha$-isoperimetric constant of $M$ is defined by
\be
 {\rm ID}_{\alpha}(M)=\inf\limits_{\Omega} \frac{\min\{\nu_+(\partial \Omega), \nu_-(\partial \Omega)\}}{m(\Omega)^{1-\frac{1}{\alpha}}}, \label{Diiso}
\ee
where the infimum is taken over all domains $\Omega\subset M$ with compact closure and smooth boundary such that $\partial \Omega \cap \partial M= \emptyset$.
\end{Def}

\begin{Def}
For a forward complete Finsler metric measure manifold $(M,F, m)$, the Dirichlet $\alpha$-Sobolev constant of $M$ is defined by
\be
{\rm SD}_{\alpha}(M)=\inf\limits_{f} \frac{\min \{\int_M F^*(df)dm, \int_M \overleftarrow{F}^{*}(df)dm\}}{\left(\int_M |f|^{\frac{\alpha}{\alpha-1}}dm \right)^{\frac{\alpha-1}{\alpha}}}, \label{DiSo}
\ee
where $f\in C^{\infty}_{0}(M)$.
\end{Def}

\begin{thm}{\label{alpha-relation}}
For $n\leq \alpha \leq \infty$, ${\rm ID}_{\alpha}(M)={\rm SD}_{\alpha}(M)$.
\end{thm}
\begin{proof}
In order to prove that ${\rm ID}_{\alpha}(M)\leq {\rm SD}_{\alpha}(M)$, it suffices to show that any Lipschitz function defined on $M$ with boundary condition $f|_{\partial M}=0$ must satisfy
\begin{equation}{\label{alpha-relation}}
\int_M F^*(df) dm \geq {\rm ID}_{\alpha}(M) \left(\int_M |f|^{\frac{\alpha}{\alpha-1}}dm\right)^{\frac{\alpha-1}{\alpha}}.
\end{equation}

Firstly, we may assume that $f\geq 0$. Let  $M_{t}:=\{x\in M \mid f(x)>t\}$ be the sublevel set of $f$. By the co-area formula (Theorem 3.3.1, \cite{Shen1}), it follows that
\begin{equation}{\label{alpha-relation-1}}
\int_M F^*(df) dm =\int_{0}^{\infty} \int_{f^{-1}(t)} d\nu dt \geq {\rm ID}_{\alpha}(M) \int_0^{\infty} m(M_t)^{\frac{\alpha-1}{\alpha}} dt,
\end{equation}
where $d\nu=d\nu_{+}$(or $d\nu_{-}$).
Let
$$
F(s):=\left(\int_0^s m(M_t)^{\frac{\alpha-1}{\alpha}}dt\right)^{\frac{\alpha}{\alpha-1}}- \frac{\alpha}{\alpha-1} \int_0^s t^{\frac{1}{\alpha-1}} m(M_t) dt.
$$
Then
$$
F'(s)= \frac{\alpha}{\alpha-1} \left(\int_0^s m(M_t)^{\frac{\alpha-1}{\alpha}}dt\right)^{\frac{1}{\alpha-1}} m(M_s)^{\frac{\alpha-1}{\alpha}}
      -\frac{\alpha}{\alpha-1} s^{\frac{1}{\alpha-1}} m(M_s).
$$
Since $\int_0^s m(M_t)^{\frac{\alpha-1}{\alpha}} dt \geq s \ m(M_s)^{\frac{\alpha-1}{\alpha}}$ for $t\leq s$, then we have $F'(s)\geq 0$. Hence, $F(s)\geq 0$. Moreover, by (\ref{alpha-relation-1}), we know that
$$
\int_M F^*(df) dm \geq {\rm ID}_{\alpha}(M) \left(\frac{\alpha}{\alpha-1} \int_{0}^{\infty} t^{\frac{1}{\alpha-1}} m(M_t) dt\right)^{\frac{\alpha-1}{\alpha}}.
$$
Further, by using integrating by parts and using the co-area formula again,  we obtain
\beqn
\frac{\alpha}{\alpha-1} \int_{0}^{\infty} t^{\frac{1}{\alpha-1}} m(M_t) dt &=& \int_{0}^{\infty} (t^{\frac{\alpha}{\alpha-1}})' \left[\int_{t}^{\infty} \left(\int_{f^{-1}(s)} \frac{1}{F(\nabla f)}d\nu\right)ds\right]dt\\
&=& \int_{0}^{\infty} t^{\frac{\alpha}{\alpha-1}}\left( \int_{f^{-1}(t)} \frac{1}{F(\nabla f)}d\nu\right)dt =\int_{0}^{\infty} \int_{f^{-1}(t)} f^{\frac{\alpha}{\alpha-1}} \frac{1}{F(\nabla f)} d\nu d t \\
&=& \int_{M} f^{\frac{\alpha}{\alpha-1}} dm.
\eeqn
This means that (\ref{alpha-relation}) holds for $f\geq 0$.

Likewise, when $f\leq 0$, letting $g=-f$, one can show that
\begin{equation*}
\int_M \overleftarrow{F}^{*}(df) dm=\int_M F^{*}(dg) dm \geq {\rm ID}_{\alpha}(M) \left(\int_M |g|^{\frac{\alpha}{\alpha-1}}dm\right)^{\frac{\alpha-1}{\alpha}}={\rm ID}_{\alpha}(M) \left(\int_M |f|^{\frac{\alpha}{\alpha-1}}dm\right)^{\frac{\alpha-1}{\alpha}}.
\end{equation*}
Thus, we get ${\rm ID}_{\alpha}(M) \leq {\rm SD}_{\alpha}(M)$.

Now we prove that ${\rm ID}_{\alpha}(M)\geq {\rm SD}_{\alpha}(M)$. Let $\Omega$ be any domain of $M$  with compact closure and smooth boundary such that $\partial \Omega \cap \partial M= \emptyset$. For $\varepsilon >0$, define the $\varepsilon$-neighborhood $N_{\varepsilon}:= \{x\in \Omega \mid d(\partial \Omega, x) \geq \varepsilon\}$. Let
$$
f_{\varepsilon}(x)=
\begin{cases}
0 & {\rm on} \ \ M\setminus \Omega, \\ \frac{1}{\varepsilon} d(\partial \Omega, x) &  {\rm on} \ \ \Omega\setminus N_{\varepsilon}, \\ 1 & {\rm on} \ \ N_{\varepsilon}.
\end{cases}
$$
Clearly, $f_{\varepsilon}(x)$ is a Lipschitz function defined on $M$ with Dirichlet boundary condition. Note that $\nabla d(\partial \Omega, x)|_{\partial \Omega}=\mathbf{n}_-$. Then we have from co-area formula
\be
\int_{M} F(\nabla f_{\varepsilon}) dm= \frac{1}{\varepsilon} \int^{\varepsilon}_0 \int_{d(\partial \Omega, x)=t} d\nu_- dt. \label{co-area}
\ee
Applying the definition of Sobolev constant to $f_{\varepsilon}$ gives
$$
\int_M F(\nabla f_{\varepsilon}) dm \geq {\rm SD}_{\alpha}(M)\left(\int_M |f_{\varepsilon}|^{\frac{\alpha}{\alpha-1}}\right)^{\frac{\alpha-1}{\alpha}} \geq {\rm SD}_{\alpha}(M) m(N_{\varepsilon})^{\frac{\alpha-1}{\alpha}}.
$$
From this and (\ref{co-area}), letting $\varepsilon\rightarrow 0$ yields
$$
\nu_{-}(\partial \Omega)\geq {\rm SD}_{\alpha}(M) m(\Omega)^{\frac{\alpha-1}{\alpha}},
$$
which implies that  ${\rm ID}_{\alpha}(M)\geq {\rm SD}_{\alpha}(M)$.

Similarly, we define the $\varepsilon$-neighborhood  $\bar{N}_{\varepsilon}:=\{x\in \Omega^c \mid d(\partial \Omega, x) \geq \varepsilon\}$, where $ \Omega^{c}:=M\setminus \Omega$. Let
$$
\bar{f}_{\varepsilon}(x)=
\begin{cases}
 1& {\rm on} \ \ \Omega, \\ 1-\frac{1}{\varepsilon} d(\partial \Omega, x) &  {\rm on} \ \ \Omega^c\setminus \bar{N}_{\varepsilon}, \\ 0 & {\rm on} \ \ \bar{N}_{\varepsilon}.
\end{cases}
$$
In this case, $\nabla d(\partial \Omega, x)|_{\partial \Omega}=\mathbf{n}_{+}$. Then, by co-area formula, we have
$$
\int_{M} \overleftarrow{F}^{*}(d\bar{f}_{\varepsilon}) dm= \frac{1}{\varepsilon} \int^{\varepsilon}_{0} \int_{d(\partial \Omega, x)=t} d\nu_{+} \ dt \longrightarrow \nu_{+}(\partial \Omega) \ (\varepsilon\rightarrow 0).
$$
Applying the definition of Sobolev constant to $\bar{f}_{\varepsilon}$ and letting $\varepsilon\rightarrow 0$  yield
$$
\int_{M} \overleftarrow{F}^{*}(d\bar{f}_{\varepsilon}) dm \geq {\rm SD}_{\alpha}(M) \left(\int_M |\bar{f}_{\varepsilon}|^{\frac{\alpha}{\alpha-1}}\right)^{\frac{\alpha-1}{\alpha}}  \longrightarrow  {\rm SD}_{\alpha}(M) m(\Omega)^{\frac{\alpha-1}{\alpha}},
$$
which implies that $\nu_{+}(\partial \Omega)\geq {\rm SD}_{\alpha}(M) m(\Omega)^{\frac{\alpha-1}{\alpha}}$.  Thus, it follows from the arbitrary of the $\Omega$ that ${\rm ID}_{\alpha}(M)\geq {\rm SD}_{\alpha}(M)$.
\end{proof}

For convenience, we normalize the  Dirichelet  $\alpha$-isoperimetric and $\alpha$-Sobolev constants as followings:
$$
{\rm ID}_{\alpha}^{*}(M)= \frac{{\rm ID}_{\alpha}(M)}{m(M)^{1/\alpha}}, \ \ \ {\rm SD}_{\alpha}^{*}(M)= \frac{{\rm SD}_{\alpha}(M)}{m(M)^{1/\alpha}}.
$$

In order  to prove $n$-Dirichlet isoperimetric constant estimate (or $n$-Dirichlet Sobolev constant estimate), we need some fundamental results.  Firstly, following the arguments of Lemma 4.1 in  \cite{DWZ} and Appendix C in \cite{Gromov}, we can prove the following proposition.

\begin{prop}{\label{W}}
Let ($M$, $F$, $m$) be an n-dimensional forward complete Finsler metric measure manifold and $\Sigma$ be any hypersurface dividing $M$ into two parts $M_1$, $M_2$. For any Borel subsets $W_i\subset M_i$ $(i=1,2)$, there exists $x_1$ in one of $W_i$ (denoted by $\widetilde{W}_1$) and a subset $W$ in the other one (denoted by $\widetilde{W}_2$), and a unique minimal geodesic from $x_1$ to any $x_2\in W$ which intersects $\Sigma$ at some $z$ such that
$$
m(W)\geq \frac{1}{2} m(\widetilde{W}_2)
$$
and
$$
d(x_1, z)\geq d(z, x_2) \ \text {or} \ \overleftarrow{d}(x_1, z)\geq \overleftarrow{d}(z, x_2).
$$
\end{prop}
\begin{proof}
We consider the product space $W_1\times W_2$ with the product measure $dm_{\times}=dm\times dm$.
Let
\beqn
V_1 &:=&\left\{(x, x')\in W_1 \times W_2 \mid d(x, z)\geq d(z, x')\right\},\\
V_2 &:=&\left\{(x, x')\in W_1 \times W_2 \mid d(x, z)\leq d(z, x')\right\}.
\eeqn
Since $m_{\times}(V_1\cup V_2)=m_{\times}(W_1\times W_2)$, then $m_{\times}(V_1)\geq \frac{1}{2}m_{\times}(W_1\times W_2)$ or $m_{\times}(V_2)\geq \frac{1}{2}m_{\times}(W_1\times W_2)$.

Suppose that $m_{\times}(V_1)\geq \frac{1}{2}m_{\times}(W_1\times W_2)$. Note that
$$
m_{\times}(V_1)=\int_{x\in W_1} dm(x) \int_{\{x'\in W_2 \mid d(x, z)\geq d(z, x')\}} dm(x') = \int_{x\in W_1}m(W_{x}) dm(x),
$$
where $W_{x}:=\left\{x' \in W_2 \mid d(x, z)\geq d(z, x')\right\} \subset W_{2}$. This means that there exists a point $x_{1}\in W_{1}=: \widetilde{W}_{1}$ such that $m(W_{x_1}) \geq \frac{1}{2}m(W_2) $. Now, let $W:=W_{x_1}\subset \widetilde{W}_{2}:= W_{2}$. It is obvious that $m(W)\geq \frac{1}{2} m(\widetilde{W}_2)$ and for any $x_{2}\in W$, we have $d(x_1, z)\geq d(z, x_{2})$.

Similarly, suppose that $m_{\times}(V_2)\geq \frac{1}{2}m_{\times}(W_1\times W_2)$. By the following fact
$$
m_{\times}(V_2)=\int_{x'\in W_2} dm(x') \int_{\{x\in W_1 \mid d(x, z)\leq d(z, x')\}} dm (x)=\int_{x'\in W_2} m(W_{x'}) dm (x'),
$$
where $W_{x'}:=\left\{x\in W_1 \mid d(x, z)\leq d(z, x')\right\}\subset W_{1}$, we know that there exists a point $x_1\in W_{2}:= \widetilde{W}_{1}$ such that $m(W_{x_1}) \geq \frac{1}{2}m(W_1)$. Now, let $W:=W_{x_1}\subset \widetilde{W}_{2}:= W_{1}$. It is obvious that $m(W)\geq \frac{1}{2} m(\widetilde{W}_2)$ and for any $x_{2}\in W$, we have $d(x_2, z) \leq d(z, x_{1})$. Furthermore, note that $\gamma_{x_{1} x_{2}}$ is minimal geodesic of $F$ from $x_1$ to $x_2$ iff $\overleftarrow{\gamma}_{x_{2} x_{1}}$ is minimal geodesic of $\overleftarrow{F}$ from $x_2$ to $x_1$ \cite{OHTA}. Thus, $d(x_2, z) \leq d(z, x_{1})$ implies that $\overleftarrow{d}(x_1, z)\geq \overleftarrow{d}(z, x_2)$.
\end{proof}

\vskip 2mm

When $\Lambda_{F}< \infty$, the following corollary is useful for the discussions later.

\begin{cor}\label{Wcor}
Let $(M, F, m)$ be an n-dimensional forward complete Finsler metric measure manifold with $\Lambda_{F}<\infty$ and $\Sigma$ be any hypersurface dividing $M$ into two parts $M_1$, $M_2$. For any Borel subsets $W_{i}\subset M_i$ $(i =1,2)$, there exists $x_1$ in one of $W_i$ (denoted by  $\widetilde{W}_1$) and a subset $W$ in the other one (denoted by $\widetilde{W}_2$), and a unique minimal geodesic from $x_1$ to any $x_2\in W$ which intersects $\Sigma$ at some $z$ such that
$$
m(W)\geq \frac{1}{2} m(\widetilde{W}_2)
$$
and
\be
\Lambda_F^2 d(x_1, z)\geq d(z, x_2). \label{cor4.5-2}
\ee
\end{cor}

Secondly, similar to the argument of Lemma 4.2 in \cite{DWZ}, we have the following result.

\begin{prop}{\label{W-estimate}}
Let $(M,F,m)$ be an n-dimensional forward complete Finsler metric measure manifold with $\Lambda_{F}<\infty$, and $\Sigma$, $W$, $x_1$ be same as in Proposition \ref{W}. Assume that $\mathbf{S}(\nabla r)\geq -\vartheta$ along any minimal geodesic segment for some $\vartheta\geq 0$. Then for $p>\frac{n}{2}$,
$$
m(W)\leq (1+\Lambda_{F}^{2})^{n-1} \Lambda_{F}^2 e^{\vartheta \Lambda_{F}^{2}D} \left(\Lambda_{F} D \cdot \nu(\Sigma')+ C(n,p) e^{ \frac{\vartheta D}{2p}} m(B^{+}_{D}(x_1))\overline{\mathcal{K}}^{\frac{1}{2}}(p, D, \vartheta)\right),
$$
where $\nu =\nu_{+}$ or $\nu_{-}$, $C(n, p):=\left(\frac{(n-1)(2p-1)}{2p-n}\right)^{\frac{1}{2}} $, $D:=\sup\limits_{x\in W} d(x_1, x)$ and
$$
\Sigma':=\{\Sigma \ \cap \gamma_{x_1x} \ |  \ \gamma_{x_1x} \text{ is minimal geodesic from} \ x_{1} \ \text{to some} \  x\in W\}.
$$
\end{prop}
\begin{proof}
Let $I_{x_1}:=\{v\in  S_{x_1}M \mid  \gamma_{v}=\exp_{x_1}(rv) \ \text{is minimal geodesic from} \ x_1 \ \text{to some} \ x\in W\}$. Write $dm=\sigma(r, \theta) dr d\theta$ in the geodesic polar coordinate $(r, \theta)\in (0, \infty)\times S_{x_1}M$ centered at $x_1$. From (\ref{volume-2}) with $K=0$,
$$
\frac{\partial}{\partial r} \left(\frac{\sigma (r, \theta)}{e^{\vartheta r}r^{n-1}}\right)\leq \varphi \frac{\sigma(r, \theta)}{e^{\vartheta r}r^{n-1}},
$$
where $\varphi=(\Delta r-\frac{n-1}{r}-\vartheta)_+$. For any $\frac{r}{1+\Lambda_F^2}\leq t\leq r$, integrating the above inequality from $t$ to $r$ yields
\be
\sigma(r, \theta)\leq \left(\frac{r}{t}\right)^{n-1} e^{\vartheta (r-t)}\left(\sigma(t, \theta)+\int_t^r\varphi \sigma(s, \theta)ds\right)
\leq (1+\Lambda_F^2)^{n-1} e^{\vartheta r\frac{\Lambda_F^2}{1+\Lambda_F^2}} \left(\sigma(t, \theta)+\int_t^r\varphi \sigma(s, \theta)ds\right). \label{sigmaphi}
\ee

For any $v\in I_{x_1}$, let $r(\theta)$ be the radius such that $\exp_{x_1}(r(\theta)v)\in \Sigma$. By (\ref{cor4.5-2}), we have
$$
W\subset \left\{\exp_{x_1}(rv) \mid r(\theta)\leq r \leq (1+\Lambda_F^2)r(\theta)\right\}.
$$
Thus, by using (\ref{sigmaphi}), we obtain
\beqn
& & m(W)\leq  \int_{I_{x_1}}  \int_{r(\theta)}^{(1+\Lambda_{F}^2)r(\theta)} \sigma(r, \theta)  dr d\theta  \\
& & \leq  (1+\Lambda_F^2)^{n-1} e^{\vartheta \Lambda_{F}^{2} D} \int_{I_{x_1}} \int_{r(\theta)}^{(1+\Lambda_F^2)r(\theta)}  \left(\sigma(r(\theta), \theta)+\int_{r(\theta)}^{r}\varphi \sigma(s, \theta)ds\right) dr d\theta  \\
& & \leq (1+\Lambda_F^2)^{n-1} \Lambda_{F}^{2}D e^{\vartheta\Lambda_F^2 D}\int_{I_{x_1}}\sigma(r(\theta), \theta) d\theta + (1+\Lambda_F^2)^{n-1} \Lambda_F^2D e^{\vartheta\Lambda_F^2 D} \int_{I_{x_1}}\int_0^D\varphi \sigma(r, \theta)dr d\theta\\
& & \leq  (1+\Lambda_F^2)^{n-1} \Lambda_{F}^{3}D e^{\vartheta \Lambda_F^2D}  \nu(\Sigma')+ (1+\Lambda_F^2)^{n-1} \Lambda_F^2D e^{\vartheta \Lambda_F^2D}\left(\int_{B_D^+(x_1)}\varphi^{2p} dm\right)^{\frac{1}{2p}} m(B^+_{D}(x_1))^{1-\frac{1}{2p}}\\
& & \leq (1+\Lambda_F^2)^{n-1} \Lambda_{F}^{3}D e^{\vartheta \Lambda_F^2D} \nu(\Sigma')+ (1+\Lambda_F^2)^{n-1} \Lambda_F^2 e^{(\Lambda_F^2+\frac{1}{2p})\vartheta D} C(n,p) m(B^+_{D}(x_1))\overline{\mathcal{K}}^{\frac{1}{2}}(p, D, \vartheta),
\eeqn
where we have used (\ref{laplacian-2}) in the last inequality and $C(n, p):=\left(\frac{(n-1)(2p-1)}{2p-n}\right)^{\frac{1}{2}} $. Besides, we have also used the fact that $\int_{I_{x_1}}\sigma (r (\theta), \theta)d \theta \leq \Lambda_{F}\nu (\Sigma ')$ in the fourth inequality, which is easily proved by a similar argument with Lemma 4.2 in \cite{DWZ} and Lemma A.8 in \cite{KSYZ}. This completes the proof.
\end{proof}

\vskip 2mm

From Theorem \ref{doubling},  we can get a relationship between $\overline{\|{\rm Ric}_{\infty}^K\|}_{p, r_1, \vartheta}$ and $\overline{\|{\rm Ric}_{\infty}^K\|}_{p, r_2, \vartheta}$ for $0< r_{1}\leq r_{2}$, which is important for the following discussions.

\begin{lem}{\label{Ric-ineq}}
Let $(M,F,m)$ be an n-dimensional Finsler metric measure manifolds. Assume that $\mathbf{S}(\nabla r)\geq-\vartheta$ along any minimal geodesic segment for some $\vartheta\geq 0$. Given $p>\frac{n}{2}$, $\Xi>1$ and $K\in \mathbb{R}$. For $0<r_{1}\leq r_2$ ($r_{2}\leq \frac{\pi}{2\sqrt{K}}$ when $K>0$),  there exists a constant $\varepsilon=\varepsilon(n,p, \vartheta, K, r_2, \Xi)$ such that if $\overline{\|{\rm Ric}_{\infty}^{K}\|}_{p, r_2, \vartheta}<\varepsilon$, then
\beq
r_{1}^{2}\overline{\|{\rm Ric}_{\infty}^{K}\|}_{p, r_{1}, \vartheta} &\leq & \Xi^{\frac{1}{p}} e^{\frac{1}{p}\left(\vartheta+(n-1)\sqrt{|K|}\right)r_2} \left(\frac{r_1}{r_2}\right)^{2-\frac{n}{p}} r_{2}^{2} \ \overline{\|{\rm Ric}_{\infty}^{K}\|}_{p, r_2, \vartheta} \nonumber \\
&\leq & \Xi^{\frac{1}{p}} e^{\frac{1}{p}\left(\vartheta+(n-1)\sqrt{|K|}\right)r_{2}} r_{2}^{2} \ \overline{\|{\rm Ric}_{\infty}^{K}\|}_{p, r_{2}, \vartheta}. \label{relation}
\eeq
\end{lem}

\begin{proof}
By Theorem \ref{doubling}, it is obvious that
\beqn
\overline{\|{\rm Ric}_{\infty}^K\|}_{p, r_1, \vartheta} &\leq& \sup\limits_{x\in M} \left(\frac{m(B_{r_2}^+(x))}{m(B_{r_1}^+(x))}\right)^{\frac{1}{p}} \overline{\|{\rm Ric}_{\infty}^K\|}_{p, r_2, \vartheta}\\
&\leq& \Xi^{\frac{1}{p}} e^{\frac{1}{p}\left(\vartheta+(n-1)\sqrt{|K|}\right)r_2}\left(\frac{r_2}{r_1}\right)^{\frac{n}{p}}\overline{\|{\rm Ric}_{\infty}^K\|}_{p, r_2, \vartheta}.
\eeqn
This completes the proof.
\end{proof}

\vskip 2mm

From Proposition \ref{W-estimate} and Lemma \ref{Ric-ineq}, we can get the following corollary.

\begin{cor}\label{B-estimate}
Let $(M, F, m)$ be an n-dimensional forward complete Finsler metric measure manifold with $\Lambda_{F}<\infty$ and $\Sigma$ be a hypersurface dividing $M$ into two parts $M_1$ and $M_2$. Assume that $\mathbf{S}(\nabla r)\geq -\vartheta$ along any minimal geodesic segment for some $\vartheta\geq 0$. For $p>\frac{n}{2}$, $R>0$ and $\Xi>1$, there exists a constant $\varepsilon=\varepsilon(n, p, \vartheta, \Lambda_F, R, \Xi)$ such that if  $\overline{\mathcal{K}}(p, R, \vartheta) <\varepsilon$, then for a metric ball $B=B^{+}_{r}(x)$ with $r\leq \frac{R}{\Lambda_F+2}$ and $x\in M$, which is divided equally by $\Sigma$, we have
\be
m(B_{r}^{+}(x))\leq 8 (1+\Lambda_{F}^2)^{n-1}\Lambda_{F}^{3} (1+\Lambda_F) e^{\vartheta \Lambda_{F}^{2}R} \cdot r \ \nu \left(\Sigma \cap B_{(\Lambda_{F}+2)r}^{+}(x)\right). \label{corvolBr}
\ee
\end{cor}
\begin{proof}
Put $W_i=B\cap M_i$ $(i =1,2)$. Let $x_1$, $W, D, \Sigma '$ be same as in Proposition \ref{W-estimate}. Then $m(B\cap M_1)=m(B\cap M_2)\leq 2m(W)$, $D\leq (\Lambda_F+1)r$ and $\Sigma'\subset \Sigma \cap B^+_{(\Lambda_F+2)r}(x)$. By Proposition \ref{W-estimate}, we have for $D\leq (\Lambda_{F}+1)r <R $
\beq
m(B^{+}_{r}(x))&\leq & 4 (1+\Lambda_{F}^2)^{n-1} \Lambda_{F}^2 e^{\vartheta \Lambda_{F}^{2} R} \left(\Lambda_{F} (\Lambda_{F}+1)r \ \nu (\Sigma \cap B^{+}_{(\Lambda_{F}+2)r}(x))\right. \nonumber\\
& &  \left. + C(n,p) e^{\frac{\vartheta R}{2p}} m(B^{+}_{(\Lambda_{F}+2)r}(x))\overline{\mathcal{K}}^{\frac{1}{2}}(p, (\Lambda_{F}+1)r, \vartheta)\right). \label{mB+K}
\eeq

By Theorem \ref{doubling}, there exists a constant $\varepsilon_{1}$,  such that if $\overline{\mathcal{K}}(p, R, \vartheta)<\varepsilon_1$, we have
$$
m(B^{+}_{(\Lambda_{F}+2)r}(x))\leq \Xi e^{\vartheta R} (\Lambda_{F}+2)^n m(B^{+}_{r}(x)).
$$

By Lemma \ref{Ric-ineq}, for all $(\Lambda_F+1)r< R$, there exists a constant $\varepsilon_{2}$ such that if $\overline{\mathcal{K}}(p, R, \vartheta) <\varepsilon_{2}$, (\ref{relation}) holds for $K=0$.  Further, if
$$
\overline{\mathcal{K}}(p, R, \vartheta)< (\Xi e^{\vartheta R})^{-\frac{1}{p}}\left(8 (1+\Lambda_{F}^2)^{n-1} \Lambda_{F}^2 (\Lambda_{F}+2)^{n}C(n, p)\Xi e^{(\Lambda_{F}^{2}+\frac{1}{2p}+1)\vartheta R}\right)^{-2},
$$
then
$$
\overline{\mathcal{K}}(p, (\Lambda_{F}+1)r, \vartheta)<\left(8 (1+\Lambda_{F}^2)^{n-1} \Lambda_{F}^2 (\Lambda_{F}+2)^{n}C(n, p)\Xi e^{(\Lambda_{F}^{2}+\frac{1}{2p}+1)\vartheta R}\right)^{-2}
$$
by (\ref{relation}).  Now we choose   $\varepsilon$ by
$$
\varepsilon:=\min\left\{\varepsilon_{1}, \varepsilon_{2}, (\Xi e^{\vartheta R})^{-\frac{1}{p}}\left(8(1+\Lambda_F^2)^{n-1} \Lambda_{F}^{2} (\Lambda_{F}+2)^{n}C(n, p)\Xi e^{(\Lambda_{F}^{2}+\frac{1}{2p}+1)\vartheta R}\right)^{-2}\right\}.
$$
Then, if  $\overline{\mathcal{K}}(p, R, \vartheta) <\varepsilon$, we know that the second term of right hand side of (\ref{mB+K}) $\leq \frac{1}{2}m(B^+_r(x))$, from which,  we obtain (\ref{corvolBr}).
\end{proof}

\vskip 2mm

\begin{lem}{\label{1-2-volume}}
Let $(M,F,m)$ be an n-dimensional forward complete Finsler metric measure manifold with $\Lambda_{F}<\infty$. Assume that $\mathbf{S}(\nabla r)\geq -\vartheta$ along any minimal geodesic segment for some $\vartheta\geq 0$. For $p>\frac{n}{2}$, there exist constants  $0<r_{0}=r_{0}(n, \Lambda_{F})<1$ and  $\varepsilon=\varepsilon(n, p, \vartheta, 1/r_{0})>0$ such that if $\overline{\mathcal{K}}(p, 1/r_{0}, \vartheta) <\varepsilon$, then for any $r_{1}\leq \frac{1}{r_{0}}$
\be
\frac{m(B_{r_{0}r_{1}}^{+}(x))}{m(B_{r_1}^{+}(x))}\leq \frac{1}{2}, \ \ \forall x\in M. \label{DWZth3.2}
\ee
\end{lem}
\begin{proof}
Choose an integer $k>0$ and $k$ numbers $\{r_i\}_{i=1}^{k}$ with $r_{i+1}<\frac{r_i}{\Lambda_F(\Lambda_{F}+2)}$. Further, let $d_{i}:= \frac{r_{i}-\Lambda_{F}r_{i+1}}{\Lambda_{F}+1}$, $i=1, 2, \ldots, k$.
Now, for any $x\in M$, choose $k$ points $x_{i}\in B^{-}_{r_i}(x)$  satisfying $d_{i}=d(x_{i}, x)$.  Then it is easy to see that
$$
B^{+}_{r_{i+1}}(x)\subset B^{+}_{d_{i}+r_{i+1}}(x_i)\setminus B^{+}_{d_{i}-\Lambda_{F}r_{i+1}}(x_i) \subset B^{+}_{d_{i}+r_{i+1}}(x_i) \subset B^{+}_{r_{i}}(x).
$$
Moreover,
$$
\frac{m(B^{+}_{r_{i+1}}(x))}{m(B^{+}_{r_{i}}(x))} \leq \frac{m(B^{+}_{r_{i+1}}(x))}{m(B^{+}_{d_{i}+r_{i+1}}(x_i))} \leq 1-\frac{m(B^{+}_{d_{i}-\Lambda_{F}r_{i+1}}(x_i))}{m(B^{+}_{d_{i}+r_{i+1}}(x_i))}.
$$
Without loss of generality, assume that we have chosen a number $r_{0} \in (0,1)$  such that  $d_{i}+\Lambda_{F}r_{i+1}< r_{1}\leq \frac{1}{r_0}$. Then, by Theorem \ref{doubling} with $\Xi=\frac{3}{2}$, we know that there exists a constant $\varepsilon=\varepsilon(n, p, \vartheta, 1/r_{0})>0$, such that if $\overline{\mathcal{K}}(p, 1/r_{0}, \vartheta) <\varepsilon$, then for $r_{1}\leq \frac{1}{r_{0}}$, we have
$$
\frac{m(B^+_{d_i-\Lambda_Fr_{i+1}}(x_i))}{m(B^+_{d_i+r_{i+1}}(x_i))} \geq \frac{2}{3}\left(\frac{d_i-\Lambda_Fr_{i+1}}{d_i+\Lambda_Fr_{i+1}}\right)^{n} e^{-\vartheta r_1}.
$$
Now we choose $a_{0}(n)$ $(<1)$ such that
$$
\frac{1-a_{0}}{1+a_{0}}=\left(\frac{3}{4}\right)^{\frac{1}{n}}.
$$
Then for any $\Lambda_{F}r_{i+1}\leq a_{0}(n)\frac{r_i}{\Lambda_{F}+2}$, we have $r_{k} \leq \left(\frac{a_{0}(n)}{\Lambda_{F}(\Lambda_{F}+2)}\right)^{k-1}r_{1}$ and
$$
\left(\frac{d_{i}-\Lambda_{F}r_{i+1}}{d_{i}+\Lambda_{F}r_{i+1}}\right)^{n} \geq \left(\frac{1-a_{0}}{1+a_{0}}\right)^{n} = \frac{3}{4}.
$$
Hence, by choosing an appropriate integer $k$, we can get
$$
\frac{m(B^{+}_{r_{k}}(x))}{m(B^{+}_{r_{1}}(x))} = \prod\limits_{i=1}^{k-1} \frac{m(B^{+}_{r_{i+1}}(x))}{m(B^{+}_{r_{i}}(x))} \leq \left(1-\frac{1}{2}e^{-\vartheta r_{1}}\right)^{k-1} \leq \frac{1}{2}.
$$
Obviously,  the proof is completed by choosing $r_{0}=\left(\frac{a_0(n)}{\Lambda_F(\Lambda_F+2)}\right)^{k-1}$ and $r_{k}=r_{0}r_{1}$.
\end{proof}

Now, based on the above discussions, we are in the position to prove the following Dirichlet isoperimetric constants estimate.

\begin{thm}{\label{constant}}
Let $(M, F, m)$ be an n-dimensional forward complete Finsler metric measure manifold with $\Lambda_{F}<\infty$. Assume that $\mathbf{S}(\nabla r)\geq-\vartheta$ along any minimal geodesic segment for some $\vartheta \geq0$. For $p>\frac{n}{2}$ and $\Xi>1$, there exist constants $0<r_{0}=r_{0}(n, \Lambda_{F})<1$ and  $\varepsilon =\varepsilon(n, p, \vartheta, \Lambda_{F}, r_{0}, \Xi)>0$ such that if $\overline{\mathcal{K}}(p, 1/r_{0}, \vartheta)<\varepsilon$, then for any $x\in M$, $r\leq 1$ with $\partial B_{1}^{+}(x)\neq \emptyset$, the normalized Dirichlet isoperimetric constant has the estimate
\be
{\rm ID}^{*}_{n}(B^{+}_{r}(x))\geq C(n, \vartheta,\Lambda_{F}, r_{0}, \Xi)^{-1} r^{-1}, \label{Dieste}
\ee
where $C(n, \vartheta,\Lambda_{F}, r_{0}, \Xi):= 10^{n+1}(1+\Lambda_{F}^2)^{n-1}(\Lambda_{F}+2)^{2n+1}\Xi \ e^{\frac{\vartheta}{r_0}(2\Lambda_{F}^{2}+\frac{1}{n})}r_{0}^{-1}$.
\end{thm}

\begin{proof}
First of all, by Lemma \ref{1-2-volume}, there exist constants  $0<r_{0}=r_{0}(n, \Lambda_{F})<1$ and $\varepsilon_{1}=\varepsilon_{1}(n, p, \vartheta, 1/r_{0})>0$ such that if $\overline{\mathcal{K}}(p, 1/r_{0}, \vartheta)<\varepsilon_{1}$ then for any $r_{1}\leq \frac{1}{r_0}$,
\begin{equation}\label{y-volume}
\frac{m(B^{+}_{(\Lambda_{F}+1)r_{0}r_{1}}(z))}{m(B^{+}_{\frac{r_1}{5\Lambda_{F}(\Lambda_{F}+2)}}(z))}\leq \frac{1}{2}, \ \  \forall z\in M.
\end{equation}
In fact, by using numbers $\{r_{i}\}_{i=1}^{k}$ chosen in the proof of Lemma {\ref{1-2-volume}}, we can choose $k$ new numbers $\tilde{r}_{i}=\frac{r_i}{5\Lambda_{F}(\Lambda_{F}+2)}$, $1 \leq i \leq k$. Clearly, we still have $\tilde{r}_{i+1}<\frac{\tilde{r}_{i}}{\Lambda_{F}(\Lambda_{F}+2)}$ for $1\leq i \leq k-1$. Then, following the argument of the  proof of Lemma {\ref{1-2-volume}}, we can choose
$r_{0}=\left(\frac{a_{0}(n)}{\Lambda_{F}(\Lambda_{F}+2)}\right)^{k-1}\frac{1}{5\Lambda_{F}(\Lambda_{F}+2)(\Lambda_{F}+1)}$ and $\tilde{r}_{k}=\left(\frac{a_{0}(n)}{\Lambda_{F}(\Lambda_{F}+2)}\right)^{k-1}\tilde{r}_{1}=(\Lambda_{F}+1)r_{0}r_{1}$, such that
$$
\frac{m(B^{+}_{\tilde{r}_{k}}(z))}{m(B^{+}_{\tilde{r}_{1}}(z))}  \leq \frac{1}{2},
$$
that is, (\ref{y-volume}) holds.

Now given any $x\in M$, let $\Omega$ be a smooth subdomain of $B^{+}_{r_{0}r_{1}}(x)$. We may assume that $\Omega$ is connected and its boundary $\Sigma=\partial \Omega$ divides $M$ into two parts $\Omega$ and $\Omega^{c}=M\setminus \Omega$. For any $z \in \Omega$, let $r_z$ be the smallest radius such that
$$
m(B_{r_z}^{+}(z)\cap \Omega)=m(B_{r_z}^{+}(z)\cap \Omega^c)=\frac{1}{2}m(B_{r_z}^{+}(z)).
$$
Since $\Omega\subset B^{+}_{(\Lambda_{F}+1)r_{0}r_{1}}(z)$ and $m(B^{+}_{(\Lambda_{F}+1)r_{0}r_{1}}(z))\leq \frac{1}{2}m(B^{+}_{\frac{r_1}{5\Lambda_{F}(\Lambda_{F}+2)}}(z))$ by (\ref{y-volume}),
we have $r_{z}\leq \frac{r_1}{5\Lambda_{F}(\Lambda_{F}+2)}<\frac{r_1}{\Lambda_{F}+2}$. Further, by Lemma \ref{Ric-ineq}, for any $r_{1}\leq \frac{1}{r_0}$, there exists a $\varepsilon_{2} =\varepsilon_{2}(n, p, \vartheta, 1/r_{0}, \Xi) >0$ such that if $\overline{\mathcal{K}}(p, 1/r_{0}, \vartheta)<\Xi^{-\frac{1}{p}} e^{-\frac{\vartheta}{pr_0}}\varepsilon_2$,
then $\overline{\mathcal{K}}(p, r_1, \vartheta) < \varepsilon_2$ by (\ref{relation}). Hence, by Corollary \ref{B-estimate}, if $\overline{\mathcal{K}}(p, 1/r_0, \vartheta)<\Xi^{-\frac{1}{p}} e^{-\frac{\vartheta}{pr_0}}\varepsilon_2$, then for any $r_{z}< \frac{r_1}{\Lambda_{F}+2}\leq \frac{1}{(\Lambda_{F}+2)r_0}$, we have
\begin{equation}\label{constant-1}
m(B_{r_z}^{+}(z))\leq 8 (1+\Lambda_{F}^2)^{n-1}\Lambda_{F}^{3} (\Lambda_{F}+1) e^{\vartheta \Lambda_{F}^{2}\frac{1}{r_0}} r_{z} \ \nu(\Sigma \cap B_{(\Lambda_{F}+2)r_z}^{+}(z)),
\end{equation}
where $\nu =\nu_{+}$ or $\nu_{-}$. Furthermore, because of $\Omega\subset \bigcup\limits_{z\in \Omega}B^{+}_{(\Lambda_{F}+2)r_z}(z)$, by Vitali Covering Lemma (see Lemma 3.4 of \cite{Simon}), we can choose a countable family of disjoint balls $B_{i}:=B^{+}_{(\Lambda_{F}+2)r_{z_i}}(z_i)$  from collection $\{B^{+}_{(\Lambda_{F}+2)r_{z}}(z)\}_{z\in \Omega}$ such that $\Omega \subset \cup_{i}B_{5\Lambda_{F}(\Lambda_{F}+2)r_{z_i}}(z_i)$. Moreover, by Theorem \ref{doubling}, there exists $\varepsilon_{3}= \varepsilon_{3}(n, p, \vartheta, 1/r_{0}, \Xi)>0$ such that if $\overline{\mathcal{K}}(p, 1/r_0, \vartheta)<\varepsilon_3$, then for $5\Lambda_{F}(\Lambda_{F}+2)r_{z}\leq r_{1}\leq \frac{1}{r_0}$, we have
$$
\frac{m(B^{+}_{5\Lambda_{F}(\Lambda_{F}+2)r_{z}}(z))}{m(B^{+}_{r_{z}}(z))} \leq \Xi e^{\frac{\vartheta}{r_0}} \left(5\Lambda_{F}\right)^{n}(\Lambda_{F}+2)^n.
$$
Hence,
\be
m(\Omega)\leq \sum\limits_{i} m(B_{5\Lambda_{F}(\Lambda_{F}+2)r_{z_i}}(z_i)) \leq \Xi e^{\frac{\vartheta}{r_0}}\left(5\Lambda_{F}\right)^{n}(\Lambda_{F}+2)^{n}\sum\limits_{i} m(B^{+}_{r_{z_i}}(z_i)). \label{main-1}
\ee
At the same time, from (\ref{constant-1}), we have
\be
\nu(\partial\Omega)\geq \sum\limits_{i} \nu(\Sigma \cap B_{i})\geq 2^{-3}(1+\Lambda_{F}^2)^{1-n}\Lambda_{F}^{-3} (\Lambda_{F}+1)^{-1} e^{-\vartheta\Lambda_{F}^{2}\frac{1}{r_0}} \sum\limits_{i} r_{z_i}^{-1} m(B^{+}_{r_{z_i}}(z_i)). \label{main-2}
\ee
The estimates (\ref{main-1}) and (\ref{main-2}) lead to
\beqn
\frac{\nu(\partial \Omega)}{m(\Omega)^{\frac{n-1}{n}}} &\geq&  \Xi^{\frac{1}{n}-1}2^{-3}5^{1-n} (1+\Lambda_F^2)^{1-n} (\Lambda_{F}+2)^{-2n-2}e^{-\vartheta\frac{2\Lambda_F^2}{r_0}}\frac{\sum_i r_{z_i}^{-1} m(B^{+}_{r_{z_i}}(z_i))}{\left(\sum_{i}m(B^{+}_{r_{z_i}}(z_i))\right)^{\frac{n-1}{n}}}\\
&\geq& \Xi^{\frac{1}{n}-1}2^{-3}5^{1-n} (1+\Lambda_F^2)^{1-n} (\Lambda_{F}+2)^{-2n-2}e^{-\vartheta\frac{2\Lambda_F^2}{r_0}}\inf_{i} \frac{ r_{z_i}^{-1} m(B^{+}_{r_{z_i}}(z_i))}{\left(m(B^{+}_{r_{z_i}}(z_i))\right)^{\frac{n-1}{n}}}\\
&=& \Xi^{\frac{1}{n}-1}2^{-3}5^{1-n} (1+\Lambda_F^2)^{1-n} (\Lambda_{F}+2)^{-2n-2}e^{-\vartheta\frac{2\Lambda_F^2}{r_0}} \inf_i \frac{m(B^{+}_{r_{z_i}}(z_i))^{\frac{1}{n}}}{r_{z_i}}.
\eeqn
On the other hand, since $B^{+}_{r_{0}r_{1}}(x)\subset B^{+}_{(\Lambda_{F}+1)r_{0}r_{1}}(z_i)$, by Theorem \ref{doubling} again, we have
$$
m(B^{+}_{r_{0}r_1}(x))
\leq m(B^{+}_{(\Lambda_{F}+1)r_{0}r_1}(z_i))
\leq \frac{1}{2}m(B^{+}_{\frac{r_1}{5\Lambda_{F}(\Lambda_{F}+2)}}(z_i))
\leq \frac{ \Xi \ r_{1}^{n} e^{\frac{\vartheta}{r_0}}}{2\cdot 5^{n}\Lambda_{F}^{n}(\Lambda_{F}+2)^n r_{z_i}^n} m(B^{+}_{r_{z_i}}(z_i)).
$$
Thus, by choosing $\varepsilon =\min\{\varepsilon_1, \Xi^{-\frac{1}{p}} e^{-\frac{\vartheta}{pr_0}}\varepsilon_2, \varepsilon_3\}$,  we can obtain the following inequality if $\overline{\mathcal{K}}(p, 1/r_{0}, \vartheta)<\varepsilon$
\be
\frac{\nu(\partial \Omega)}{m(\Omega)^{\frac{n-1}{n}}}
\geq \left( 10^{n+1}(1+\Lambda_{F}^2)^{n-1}(\Lambda_{F}+2)^{2n+1}\Xi\right)^{-1}e^{-\frac{\vartheta}{r_0}(2\Lambda_{F}^{2}+\frac{1}{n})}r_{1}^{-1} m(B^{+}_{r_{0}r_{1}}(x))^{\frac{1}{n}}. \label{main-3}
\ee

Now, for any  $r\leq 1$, set $r_1=\frac{r}{r_0}\leq \frac{1}{r_0}$, then the estimate (\ref{Dieste}) is obtained from (\ref{main-3}). This completes the proof.
\end{proof}

\vskip 2mm

Combining Theorem \ref{constant} and Theorem \ref{alpha-relation}, we have

\begin{cor}{\label{Sobolev-ineq}}
Let $(M, F, m)$ be an n-dimensional forward complete Finsler metric measure manifold with $\Lambda_{F}<\infty$. Assume that $\mathbf{S}(\nabla r)\geq-\vartheta$ along any minimal geodesic segment for some $\vartheta \geq 0$. For $p>\frac{n}{2}$ and $\Xi>1$, there exist constants $0<r_{0}=r_{0}(n, \Lambda_{F})<1$ and  $\varepsilon =\varepsilon(n, p, \vartheta, \Lambda_{F}, r_{0}, \Xi)>0$ such that if $\overline{\mathcal{K}}(p, 1/r_{0}, \vartheta)<\varepsilon$, then for any $x\in M$, $r\leq 1$ with $\partial B_1^+(x)\neq \emptyset$, the following inequality holds
\be
\left(\int_{B^{+}_{r}(x)}|f|^{\frac{n}{n-1}}dm\right)^{\frac{n-1}{n}}\leq C(n, \vartheta,\Lambda_{F}, r_{0}, \Xi)  m(B_{r}^{+}(x))^{-\frac{1}{n}} r\int_{B_{r}^{+}(x)} F^{*}(df)dm, \forall f\in C_0^{\infty}(B_{r}^{+}(x)).\label{s-1}
\ee
Further, when $n >2$, applying (\ref{s-1}) to $f^{\frac{2(n-1)}{n-2}}$ and using the H\"{o}lder inequality give
\beq
\left(\int_{B^{+}_{r}(x)}|f|^{\frac{2n}{n-2}}dm\right)^{\frac{n-2}{n}}\leq \tilde{C}(n, \vartheta,\Lambda_{F}, r_{0}, \Xi) m(B_{r}^{+}(x))^{-\frac{2}{n}} r^2 \int_{B_{r}^{+}(x)} F^{*2}(df)dm, \forall f\in C_{0}^{\infty}(B_{r}^{+}(x)).\label{s-2}
\eeq
\end{cor}

\section{Applications}{\label{Application}}
In this section we will give lower bound estimate for the first (nonzero) Dirichlet eigenvalue and a local gradient estimate for the harmonic functions as the applications of the Dirichlet isoperimetric constant estimate.

\begin{thm}\label{eiges}
Let $(M, F, m)$ be an n-dimensional forward complete Finsler metric measure manifold with $\Lambda_{F}<\infty$. Assume that $\mathbf{S}(\nabla r)\geq-\vartheta$ along any minimal geodesic segment for some $\vartheta \geq 0$. For $p>\frac{n}{2}$, there exist constants $0<r_{0}=r_{0}(n, \Lambda_{F})<1$ and  $\varepsilon =\varepsilon(n, p, \vartheta, \Lambda_{F}, r_{0})>0$ such that if $\overline{\mathcal{K}}(p, 1/r_{0}, \vartheta)<\varepsilon$, then for any $x\in M$, $r\leq 1$ with $\partial B_1^+(x)\neq \emptyset$, the first (nonzero) Dirichlet eigenvalue for the Finsler Laplacian has lower bound
$$
\lambda_1(B^{+}_{r}(x))\geq \tilde{C}(n, \vartheta, \Lambda_{F}, r_{0})^{-1}r^{-2}.
$$
\end{thm}
\begin{proof} Let $\lambda_1$ be the first (nonzero) Dirichlet eigenvalue for Finsler Laplacian $\Delta$ and $f$ be its associating eigenfunction, that is,
$$
\left\{ {\begin{array}{*{20}{c}}
\begin{array}{l}
\Delta f=-\lambda_1 f, \\
f =0,
\end{array}&\begin{array}{l}
\text{in} \ B^{+}_{r}(x),\\
\text{on} \ \partial B^{+}_{r}(x).
\end{array}
\end{array}} \right.
$$
Then, by (\ref{s-2}) with $\Xi =2$, we have
$$
\int_{B^{+}_{r}(x)}f^{2} dm\leq \left(\int_{B^{+}_{r}(x)}f^{\frac{2n}{n-2}} dm\right)^{\frac{n-2}{n}}m(B^{+}_{r}(x))^{\frac{2}{n}}\leq \tilde{C}(n, \vartheta, \Lambda_{F}, r_{0}) r^2 \int_{B_{r}^{+}(x)} F^{*2}(df)dm.
$$
Thus $\lambda_{1}(B^{+}_{r}(x))\geq \tilde{C}(n, \vartheta, \Lambda_{F}, r_{0})^{-1}r^{-2}$.
\end{proof}

In the following, we always assume that $n>2$. For simplicity, let $C_{\rm SD}(\Omega)$ be the normalized local Sobolev constant of $\Omega\subset M$, namely,
\begin{equation}{\label{CSD}}
\left(\int_{\Omega}|f|^{\frac{2n}{n-2}}dm\right)^{\frac{n-2}{n}}\leq C_{\rm SD}(\Omega)\int_{\Omega} F^{*2}(df)dm, \ \ \forall f\in C_{0}^{\infty}(\Omega).
\end{equation}
Besides,  we will denote $B_{R}:=B_{R}^{+}(x_0)$ for any given $x_{0}\in M$.

 In order to derive the gradient estimate for the harmonic function,  let us first give a mean value inequality for nonnegative subsolutions of a class of elliptic equations. It is worth to mention  that the present authors have obtained a mean value inequality  for non-negative subsolutions of the elliptic equations with $f\in L^{\infty}(B_{R}^{+}(x_{0}))$  under the condition that the weighted Ricci curvature ${\rm Ric}_\infty$ has a lower bound \cite{chengF}. The following results for $f\in L^{q}(B_{R}^{+}(x_{0}))$ can be regarded as a supplement of the case with $f\in L^{\infty}(B_{R}^{+}(x_{0}))$ discussed in  \cite{chengF}.

\begin{lem}\label{mean-ineq}
Let $(M, F, m)$ be an $n$-dimensional forward complete Finsler measure space with $\Lambda_{F}<\infty$ and Sobolev inequality (\ref{CSD}) hold. Suppose that $u$ is a non-negative function on $B_{R}^{+}(x_0)$ for any $x_{0}\in M$ satisfying
$$
\Delta u \geq -fu
$$
in the weak sense, where $f\in L^{q}(B_{R}^{+}(x_0)) \ (\frac{n}{2}<q<\infty)$ is non-negative. Then, we have the following inequalities.

{\rm (1)} For any $k\geq 2$, there are constants $C_{i}(n, q)>0 \ (i=1,2)$ depending only on $n$ and $q$, such that
\begin{equation}
\sup _{B_{\delta R}^{+}(x_0)} u^k \leq  C_{2}(n, q)\left(\frac{4(4+\Lambda_F)}{(1-\delta)^2}\frac{C_{\rm SD}(B_{R}^{+}(x_0))}{ R^2}+ 8 C_{1}(n, q)C_{\rm SD}^{\beta}(B_{R}^{+}(x_0))\left(\frac{k}{2}\mathcal{A}\right)^{\beta}\right)^{\frac{n}{2}} \int_{B_{R}^{+}(x_0)} u^{k} dm  \label{meanineq-1}
\end{equation}
for any $\delta \in(0,1)$, where $\mathcal{A}:=\left(\int_{B_{R}^{+}(x_0)} f^q dm\right)^{\frac{1}{q}}$ and $\beta:=\frac{2q}{2q-n}$.

{\rm (2)} For any $0<k<2$ and $\delta \in(0,1)$, there are constants $C_{1}(n, q)$ and $C_{3}(n, q, k)>0$ depending on $n$, $q$ and $k$, such that
\be
\sup\limits_{B_{\delta R}^{+}(x_0)} u^k \leq C_{3}(n, q, k)\left(\frac{4(4+\Lambda_F)}{(1-\delta)^2}\frac{C_{\rm SD}(B_{R}^{+}(x_0))}{ R^2}+ 8C_{1}(n, q)C_{\rm SD}^{\beta}(B_{R}^{+}(x_0))\mathcal{A}^{\beta}\right)^{\frac{n}{2}}\int_{B_{R}^{+}(x_0)} u^k d m.\label{meanineq-2}
\ee

\end{lem}

\begin{proof}  Since $u$ is a nonnegative function satisfying $\Delta u\geq-fu$ in the weak sense on $B_R$, we have
\begin{equation}{\label{mean-2}}
\int_{B_{R}} d \phi(\nabla u) d m \leq \int_{B_{R}}\phi fu dm
\end{equation}
for any nonnegative function $\phi \in \mathcal{C}_0^{\infty}\left(B_R\right)$.  For any $0 < \delta<\delta^{\prime} \leq 1$ and $a\geq1$, let $\phi=u^{2a-1}\varphi^2$,  where $\varphi$ is a cut-off function defined by
$$
\varphi(x)= \begin{cases}1 & \text { on } B_{\delta R}, \\ \frac{\delta^{\prime} R-d \left(x_0, x\right)}{\left(\delta^{\prime}-\delta\right) R} & \text { on } B_{\delta^{\prime} R} \backslash B_{\delta R}, \\ 0 & \text { on } M \backslash B_{\delta^{\prime} R}.\end{cases}
$$
Then $F^{*}(-d \varphi) \leq \frac{1}{\left(\delta^{\prime}-\delta\right) R}$ and hence $F^{*}(d \varphi) \leq \frac{\Lambda_F}{\left(\delta^{\prime}-\delta\right) R}$ a.e. on $B_{\delta^{\prime} R}$. Thus, by (\ref{mean-2}), we have
$$
(2a-1)\int_{B_R} \varphi^2 u^{2a-2}F^{*2}(d u) d m  +2 \int_{B_R} \varphi u^{2a-1} d \varphi(\nabla u ) d m \leq  \int_{B_R} \varphi^{2}u^{2a} f d m.
$$
Therefore,
$$
a^2\int_{B_R} \varphi^2 u^{2a-2}F^{*2}(d u) d m \leq -2a \int_{B_R} \varphi u^{2a-1} d \varphi(\nabla u ) d m + a \int_{B_R} \varphi^2 u^{2a} f d m,
$$
that is,
\beqn
\int_{B_R} \varphi^2 F^{*2}(d u^{a}) d m &\leq & -2 \int_{B_R} \varphi u^{a} d \varphi(\nabla u^{a} ) d m + a \int_{B_R} \varphi^2 u^{2a} f d m \\
&\leq & 2 \int_{B_R} \varphi u^{a} F^{*}(-d \varphi) F\left(\nabla u^{a}\right)dm + a\int_{B_R}\varphi^2 u^{2a} f d m \\
& \leq & \frac{1}{2} \int_{B_R} \varphi^2 F^{2}(\nabla u^{a}) dm + 2\int_{B_R}u^{2a} F^{*2}(-d \varphi) dm + a \int_{B_R} \varphi^2 u^{2a} f d m.\\
\eeqn
It follows from the above inequality that
\beq
\int_{B_R} \varphi^2 F^{*2}(d u^a) dm \leq  4 \int_{B_R} u^{2a} F^{* 2}(-d \varphi) d m + 2a \int_{B_R} \varphi^2 u^{2a} f d m.\label{meanv1}
\eeq
Obviously, Sobolev inequality (\ref{CSD}) and (\ref{meanv1}) imply
\beq
C_{\rm SD}^{-1}(B_R)\left(\int_{B_R} (u^a\varphi)^{2\mu}dm\right)^{\frac{1}{\mu}}&\leq & \int_{B_R}F^{*2}(d(\varphi u^a))dm \nonumber\\
&\leq& 2\int_{B_R}u^{2a}F^{*2}(d\varphi)dm +2\int_{B_R}\varphi^2 F^{*2}(du^{a})dm \nonumber\\
&\leq& (8+2\Lambda_F)\int_{B_R}u^{2a}F^{*2}(-d\varphi)dm + 4a \int_{B_R} \varphi^2 u^{2a} f d m, \label{meanv2}
\eeq
where $\mu=\frac{n}{n-2}$. For the second term in the right hand side of (\ref{meanv2}), by H\"{o}lder inequality, we have
$$
a \int_{B_R} \varphi^2 u^{2a} f d m \leq a \mathcal{A} \left(\int_{B_R}\left(\varphi^2 u^{2a}\right)^{\frac{q}{q-1}}dm\right)^{\frac{q-1}{q}}
\leq a\mathcal{A} \left(\int_{B_R} \varphi^2 u^{2a}dm\right)^{\frac{\mu(q-1)-q}{q(\mu-1)}} \left(\int_{B_R}\left(\varphi^2 u^{2a}\right)^{\mu} dm\right)^{\frac{1}{q(\mu-1)}},
$$
where $\mathcal{A}:=\left(\int_{B_R} f^q dm \right)^{\frac{1}{q}}$. Since $w^{\varepsilon}\leq \delta^{\frac{\varepsilon-1}{\varepsilon}}w+\delta\varepsilon^{\frac{1}{1-\varepsilon}}(\frac{1}{\varepsilon}-1)$ for any $w\geq 0$, $0<\varepsilon<1$ and $\delta>0$, we choose $\varepsilon=\frac{\mu(q-1)-q}{q(\mu-1)}$ and $w=\left(a\mathcal{A}\right)^{\frac{q(\mu-1)}{\mu(q-1)-q}} \int_{B_R}\varphi^2 u^{2a}dm  \cdot \left(\int_{B_R}\left(\varphi^2 u^{2a}\right)^{\mu}dm\right)^{-\frac{1}{\mu}}$, then the following inequality holds,
\beqn
& &a\mathcal{A}\left(\int_{B_R}\varphi^2 u^{2a}dm\right)^{\frac{\mu(q-1)-q}{q(\mu-1)}} \left(\int_{B_R}\left(\varphi^2 u^{2a}\right)^{\mu}dm\right)^{\frac{\mu-q(\mu-1)}{q\mu(\mu-1)}}\\
& &\leq \delta^{\frac{\varepsilon-1}{\varepsilon}} \left(a\mathcal{A}\right)^{\frac{q(\mu-1)}{\mu(q-1)-q}} \int_{B_R}\varphi^2 u^{2a}dm  \cdot \left(\int_{B_R}\left(\varphi^2 u^{2a}\right)^{\mu}dm\right)^{-\frac{1}{\mu}} + \delta\varepsilon^{\frac{1}{1-\varepsilon}}\left(\frac{1}{\varepsilon}-1\right).
\eeqn
Multiplying the both sides of above inequality by $\left(\int_{B_R}\left(\varphi^2 u^{2a}\right)^{\mu}\right)^{\frac{1}{\mu}}$ yields
\beqn
& &a\mathcal{A}\left(\int_{B_R}\varphi^2 u^{2a}dm\right)^{\frac{\mu(q-1)-q}{q(\mu-1)}} \left(\int_{B_R}\left(\varphi^2 u^{2a}\right)^{\mu}dm\right)^{\frac{1}{q(\mu-1)}}\\
& &\leq \delta^{\frac{\varepsilon-1}{\varepsilon}} \left(a\mathcal{A}\right)^{\frac{q(\mu-1)}{\mu(q-1)-q}} \int_{B_R}\varphi^2 u^{2a}dm + \delta\varepsilon^{\frac{1}{1-\varepsilon}}\left(\frac{1}{\varepsilon}-1\right)\left(\int_{B_R} \left(\varphi^2 u^{2a}\right)^{\mu}dm\right)^{\frac{1}{\mu}}.
\eeqn
Further, choosing $\delta$ such that $\delta \varepsilon^{\frac{1}{1-\varepsilon}}\left(\frac{1}{\varepsilon}-1\right)=\frac{C^{-1}_{\rm SD}(B_R)}{8}$, namely, $\delta=\frac{C^{-1}_{\rm SD}(B_R)}{8} \varepsilon^{\frac{1}{\varepsilon-1}}\frac{\varepsilon}{1-\varepsilon}$, then we obtain
$$
a \int_{B_R} \varphi^2 u^{2a} f d m
\leq C_{1}(n, q)C_{SD}^{\frac{1-\varepsilon}{\varepsilon}}(B_R) \left(a\mathcal{A}\right)^{\frac{q(\mu-1)}{\mu(q-1)-q}} \int_{B_R}\varphi^2 u^{2a}dm + \frac{C_{\rm SD}^{-1}(B_R)}{8} \left(\int_{B_R}\left(\varphi^2u^{2a}\right)^{\mu}dm\right)^{\frac{1}{\mu}},
$$
where $C_{1}(n, q):=8^{\frac{1-\varepsilon}{\varepsilon}}\varepsilon^{\frac{1}{\varepsilon}} \left(\frac{\varepsilon}{1-\varepsilon}\right)^{\frac{\varepsilon-1}{\varepsilon}}$. Thus, (\ref{meanv2}) can be rewritten as
$$
\left(\int_{B_R} (u^{a}\varphi)^{2\mu}dm\right)^{\frac{1}{\mu}}
\leq \left(4(4+\Lambda_F)\frac{C_{\rm SD}(B_R)}{(\delta'-\delta)^2 R^2} +8C_{1}(n, q)C_{\rm SD}^{\beta}(B_R) \left(a\mathcal{A}\right)^{\beta} \right)\int_{B_{\delta' R}} u^{2a}dm,
$$
where $\beta:=\frac{q(\mu-1)}{\mu(q-1)-q}=\frac{2q}{2q-n}$. By using the above inequality and H\"{o}lder's inequality, we have
\beq
\int_{B_{\delta R}} u^{2a\left(2-\frac{1}{\mu}\right)}d m &\leq & \int_{B_R}(u^a \varphi)^{2\left(2-\frac{1}{\mu}\right)} d m \nonumber\\
&\leq & \left(\int_{B_R}(u^a \varphi)^{2\mu} d m\right)^{\frac{1}{\mu}} \cdot\left(\int_{B_R}(u^a \varphi)^2 d m\right)^{\frac{\mu-1}{\mu}} \nonumber \\
&\leq & \left(4(4+\Lambda_F)\frac{C_{\rm SD}(B_R)}{(\delta'-\delta)^{2} R^2} +8C_{1}(n, q)C_{\rm SD}^{\beta}(B_R) \left(a\mathcal{A}\right)^{\beta} \right)\left(\int_{B_{\delta'R}} u^{2a}dm\right)^{2-\frac{1}{\mu}}.\label{moser}
\eeq

Let $t:=2-\frac{1}{\mu}$ and choose $k\geq2$ and $b\geq 1$ such that $a =\frac{k}{2}b\geq 1$. Then (\ref{moser}) can be rewritten as
\begin{equation}\label{mean-3}
\int_{B_{\delta R}} u^{kb t } d m \leq \left(\frac{\mathcal{B}}{(\delta'-\delta)^2} + \mathcal{G}b^{\beta} \right)\left(\int_{B_{\delta^{\prime} R}} u^{kb } d m\right)^t,
\end{equation}
where $\mathcal{B}:=4(4+\Lambda_F)\frac{C_{\rm SD}(B_R)}{ R^2}$ and $\mathcal{G}:=8C_{1}(n, q)C_{\rm SD}^{\beta}(B_R)\left(\frac{k}{2}\mathcal{A}\right)^{\beta}$.

For any $0<\delta <1$, let $\delta_0=1$ and $\delta_{i+1}=\delta_i-\frac{1-\delta}{2^{i+1}}$ on $B_{\delta_i R}$, $i=0,1, \cdots$. Applying (\ref{mean-3}) for $\delta^{\prime}=\delta_i, \ \delta=\delta_{i+1}$ and $b=t^i$, we have
$$
\int_{B_{\delta_{i+1} R}} u^{k t^{i+1}} d m \leq \left(\frac{4^{i+1}\mathcal{B}}{(1-\delta)^2}+\mathcal{G}t^{\beta i} \right)\left(\int_{B_{\delta_i R}} u^{k t^i} d m\right)^t.
$$
By Moser's iteration, one obtains that
\beqn
\|u^k\|_{L^{t^{i+1}}(B_{\delta_{i+1} R})}= \left(\int_{B_{\delta_{i+1} R}} u^{k t^{i+1}} d m\right)^{\frac{1}{t^{i+1}}}
&\leq & \prod\limits_{j=1}^{i+1} \left(\frac{4^{j}\mathcal{B}}{(1-\delta)^2}+\mathcal{G}t^{\beta (j-1)} \right)^{t^{-j}} \int_{B_R} u^k d m\\
&\leq & \left(\frac{\mathcal{B}}{(1-\delta)^2}+\mathcal{G}\right)^{\sum\limits_{j=1}^{i+1} t^{-j}} \max\{4, t^{\beta}\}^{\sum\limits_{j=1}^{i+1} jt^{-j}}\int_{B_R} u^k d m.
\eeqn
Because $\sum_{j=1}^{\infty} t^{-j}=\frac{\mu}{\mu-1}$ and $\sum_{j=1}^{\infty} j t^{-j}$ converges, we have
\be
\|u^k\|_{L^{\infty} (B_{\delta R})} \leq  C_{2}(n, q)\left(\frac{4(4+\Lambda_F)}{(1-\delta)^2}\frac{C_{\rm SD}(B_R)}{ R^2}+ 8C_{1}(n, q)C_{\rm SD}^{\beta}(B_R)\left(\frac{k}{2}\mathcal{A}\right)^{\beta}\right)^{\frac{\mu}{\mu-1}} \int_{B_R} u^k d m,\label{meanp2}
\ee
which implies (\ref{meanineq-1}) with $k\geq 2$.

For the case when $0<k<2$, by using iteration again and similar to the argument of Theorem 1.1 in \cite{chengF}, we can obtain (\ref{meanineq-2}).  This completes the proof.
\end{proof}
\vskip 2mm

From the mean value inequalities given in Lemma \ref{mean-ineq} and Bochner-Weitzenb\"{o}ck formula (\ref{BWforinf}), we can derive the following gradient estimate.

\begin{thm}{\label{gradient}}
Let $(M, F, m)$ be an n-dimensional forward complete Finsler metric measure manifold equipped with a uniformly convex and uniformly smooth Finsler metric $F$, and Sobolev inequality (\ref{CSD}) hold. Let $u$ be a harmonic function on $B_{R}^{+}(x_0)$ for any $x_{0}\in M$. Then for $p>\frac{n}{2}$, there exists a constant $C=C(n, p, \kappa, \kappa^{*})>0$ such that
\begin{equation*}
\sup _{B_{\frac{1}{2}R}^{+}(x_0)} F^{2}(\nabla u) \leq \frac{C}{R^{2}}\left(\frac{C_{\rm SD}(B_{R}^{+}(x_{0}))}{ R^2}+ C_{\rm SD}^{\beta}(B_{R}^{+}(x_0))e^{\frac{\beta\vartheta}{p}R} m(B_{R}^{+}(x_0))^{\frac{\beta}{p}}R^{-2\beta}\overline{\mathcal{K}}^{\beta}(p, R, \vartheta)\right)^{\frac{n}{2}}\int_{B_{R}^{+}(x_0)}u^2 dm,
\end{equation*}
where $\beta:=\frac{2p}{2p-n}$.
\end{thm}

\begin{proof}
Let $u(x)$ be a harmonic function on $B_R:=B_{R}^{+}(x_0)$. Then $u\in W^{2,2}_{\rm loc}(B_R)\cap \mathcal{C}^{1, \alpha}(B_R)$ and $F^{2}(\nabla u)\in W^{1,2}_{\rm loc}(B_R)\cap \mathcal{C}^{\alpha}(B_R)$. It follows from the Bochner-Weitzenb\"{o}ck type formula (\ref{BWforinf}) and the definition of ${\rm Ric}_{\infty}^0$ that
\begin{equation}\label{geheat}
\int_{B_R} \phi \Delta^{\nabla u}F^2(\nabla u) dm \geq -2\int_{B_R} \phi {\rm Ric}_{\infty}^0 F^2(\nabla u) dm
\end{equation}
for each non-negative bounded function $\phi\in H_0^{1}(B_R) \bigcap L^{\infty}(B_R)$. By replacing $u$ by $F^2(\nabla u)$, we can find that (\ref{geheat}) is an analogue of (\ref{mean-2}) with $f=2{\rm Ric}_{\infty}^0$. Hence, following the argument of Lemma \ref{mean-ineq} and by uniform convexity and uniform smoothness of $F$, we can obtain the following inequality for $k=1$
\begin{equation}\label{F-meanineq-1}
\sup\limits_{B_{\frac{1}{2} R}} F^{2}(\nabla u)  \leq C(n, p, \kappa, \kappa^{*})\left(\frac{C_{\rm SD}(B_R)}{ R^2}+ C_{\rm SD}^{\beta}(B_R)e^{\frac{\beta\vartheta}{p}R} m(B_R)^{\frac{\beta}{p}}R^{-2\beta}\overline{\mathcal{K}}^{\beta}(p, R, \vartheta)\right)^{\frac{n}{2}}\int_{B_{\frac{2}{3}R}} F^2(\nabla u) d m,
\end{equation}
where $\beta:=\frac{2p}{2p-n}$ and $C(n, p, \kappa, \kappa^{*})$ is universal positive constant.

In the following, we continue to denote by $C$ some universal constant, which may be different line by line. Let $\varphi(x)\in C_0^{\infty}(B_R)$ be a cut-off function such that
$$
0\leq\varphi\leq1, \ \ \varphi\equiv1 \ \text {in} \ B_{\frac{2}{3}R}, \ \ \varphi\equiv0 \ \text {in} \ B_R\backslash B_{\frac{3}{4}R}, \ \  \text{and} \ F^*(-d \varphi) \leq \frac{C}{ R} \  \text{a.e. in} \ B_{\frac{3}{4} R}.
$$
It follows that
\beqn
\int_{B_R}\varphi^2 F^2(\nabla u) dm &=&\int_{B_R}\varphi^2 du(\nabla u) dm-\int_{B_R} d(\varphi^2 u )(\nabla u) dm\\
&=&-2\int_{B_R}u\varphi d\varphi(\nabla u)dm\\
&\leq &~\frac{1}{2}\int_{B_R}\varphi^2 F^2(\nabla u) dm +2\int_{B_R}u^2 F^{*2}(-d\varphi) dm,
\eeqn
that is,
$$
\int_{B_{\frac{2}{3}R}} F^2(\nabla u) dm\leq 4\int_{B_R}u^2 F^{*2}(-d \varphi) dm \leq \frac{C}{R^2} \int_{B_R}u^2 dm.
$$
Then, from (\ref{F-meanineq-1}), we obtain
\begin{equation*}
\sup _{B_{\frac{1}{2}R}} F^2(\nabla u) \leq  \frac{C(n, p, \kappa, \kappa^{*})}{R^{2}}\left(\frac{C_{\rm SD}(B_R)}{ R^2}+ C_{\rm SD}^{\beta}(B_R)e^{\frac{\beta\vartheta}{p}R} m(B_R)^{\frac{\beta}{p}}R^{-2\beta}\overline{\mathcal{K}}^{\beta}(p, R, \vartheta)\right)^{\frac{n}{2}}\int_{B_{R}}u^2 dm.
\end{equation*}
This completes the proof.
\end{proof}

From (\ref{s-2}) and Theorem \ref{gradient}, we get the following theorem directly.

\begin{thm}{\label{gradient-K}}
Let $(M, F, m)$ be an n-dimensional forward complete Finsler metric measure manifold equipped with a uniformly convex and uniformly smooth Finsler metric $F$. Assume that $\mathbf{S}(\nabla r)\geq-\vartheta$ along any minimal geodesic segment for some $\vartheta \geq 0$. Let $u$ be a harmonic function on $B_{R}=B_{R}^{+}(x_0)$. For $p>\frac{n}{2}$, there exist constants $0<r_{0}=r_{0}(n, \Lambda_{F})<1$, $\varepsilon=\varepsilon(n, p, \vartheta, \kappa, \kappa^{*}, r_{0})>0$ and $C=C(n, p, \kappa, \kappa^{*}, \vartheta , r_{0})>0$ such that if $\overline{\mathcal{K}}(p, 1/r_{0}, \vartheta)<\varepsilon$, then for any $x_{0}\in M$, $R\leq 1$ with $\partial B_{1}^{+}(x_0)\neq \emptyset$, the following holds
\begin{equation*}
\sup _{B_{\frac{1}{2}R}^{+}(x_0)} F^{2}(\nabla u) \leq \frac{C}{R^{2}}m(B_R)^{-1}\|u\|_{L^2(B_{R}^{+}(x_{0}))}^2.
\end{equation*}
\end{thm}


\begin{thebibliography}{Ma}

\bibitem{BaoChSh}  Bao, D., Chern, S.S., Shen, Z.: An Introduction to Riemann-Finsler Geometry, GTM 200. Springer, New York (2000)

\bibitem{ChSh} Cheng,  X., Shen, Z.: Some inequalities on Finsler manifolds with weighted Ricci curvature bounded below. Results Math. {\bf 77}: 70 (2022)

\bibitem{chengF} Cheng, X. and Feng, Y.: Some functional inequalities and their applications on Finsler measure spaces. J. Geom. Anal. {\bf 34}: 127 (2024)

\bibitem{Chern} Chern, S.S.:  Finsler geometry is just Riemannian geometry without the quadratic restriction. Notices of the American Mathematical Society, {\bf 43}(9), 959-963 (1996)

\bibitem{ChernShen}  Chern, S.S., Shen, Z.: Riemann-Finsler Geometry, Nankai Tracts in Mathematics, Vol. 6. World Scientific, Singapore (2005)

\bibitem{DWZ} Dai, X., Wei, G., Zhang, Z.: Local Sobolev constant estimate for integral Ricci curvature bounds. Adv. Math. {\bf 325}, 1-13 (2018)

\bibitem{Ga} Gallot S., Isoperimetric inequalities based on integral norms of Ricci curvature. Soc. Math. de France, Ast\'{e}risque, {\bf 157-158}, 191-216 (1988)

\bibitem{Gromov} Gromov, M.: Paul Levy's isoperimetric inequality, Appendix C in Metric Structures for Riemannian and Non-Riemannian Spaces. Progr. Math. {\bf 152}, Birkhauser (2001)

\bibitem{KSYZ}  Krist\'{a}ly, A., Shen, Z., Yuan, L., Zhao, W.:  Nonlinear spectrums of Finsler manifolds. Math. Z. {\bf  300}, 81-123 (2022)

\bibitem{KristalyZ} Krist\'{a}ly, A., Zhao, W.: On the geometry of irreversible metric-measure spaces: Convergence, stability and analytic aspects. J. Math. Pures  Appl. {\bf 158}, 216-292 (2022)

\bibitem{Ohta} Ohta, S.:  Uniform convexity and smoothness, and their applications in Finsler geometry. Math. Ann. {\bf 343}, 669-699 (2009)

\bibitem{OHTA} Ohta, S.:  Comparison Finsler Geometry. Springer Monographs in Mathematics, Springer, Cham (2021)

\bibitem{OS2} Ohta, S., Sturm, K.-T.: Bochner-Weizenb\"{o}ck formula and Li-Yau estimates on Finsler manifolds. Adv. Math. {\bf 252}, 429-448 (2014)

\bibitem{PPS} Petersen, P., Sprouse, C.: Integral curvature bounds, distance estimates and applications. J. Diff. Geom. {\bf 50}, 269-298 (1998)

\bibitem{PeterW1} Petersen, P., Wei, G.: Relative volume comparison with integral curvature bounds. Geom. Funct. Anal. {\bf 7}, 1031-1045 (1997)

\bibitem{PeterW2} Petersen, P., Wei, G.: Analysis and geometry on manifolds with integral Ricci curvature bounds II. Trans. Amer. Math. Soc. {\bf 353}(2), 457-478 (2000)

\bibitem{Ra}  Rademacher, H.B.: Nonreversible Finsler metrics of positive flag curvature. In: ``A Sampler of Riemann-Finsler Geometry",  MSRI Publications, {\bf 50}. Cambridge University Press, Cambridge (2004)

\bibitem{Simon} Simon, L.:   Introduction to Geometric Measure Theory. Tsinghua Lectures (2014)

\bibitem{ShenZ} Shen, Y., Zhao, W.: Gromov pre-compactness theorems for nonreversible Finsler manifolds. Diff. Geom. Appl. {\bf 28}, 565-581 (2010)

\bibitem{shen} Shen, Z.:  Volume comparison and its applications in Riemann-Finsler geometry. Adv. Math. {\bf 128}(2), 306-328 (1997)

\bibitem{Shen1} Shen, Z.:  Lectures on Finsler Geometry.  World Scientific, Singapore (2001)

\bibitem{Wang-Wei} Wang, L., Wei G.: Local Sobolev constant estimate for integral Bakry-\'{E}mery Ricci curvature. Pac. J. Math. {\bf 300}(1), 233-256 (2019)

\bibitem{WuB2012} Wu, B.-Y.: On integral Ricci curvature and topology of Finsler manifolds. Int. J. Math. {\bf 23}(11): 1250111 (2012)

\bibitem{WuB2021} Wu, B.-Y.: Volume growth of Finsler manifolds with integral Ricci curvature bound. Results Math. {\bf 76}: 212 (2021)

\bibitem{WuXin} Wu, B.-Y., Xin, Y.:  Comparison theorems in Finsler geometry and their applications. Math. Ann. {\bf 337}, 177-196 (2007)

\bibitem{Xia} Xia, Q.: Geometric and functional inequalities on Finsler manifolds. J. Geom. Anal. {\bf 30}, 3099-3148 (2020)

\bibitem{Yang1} Yang, D.: Convergence of Riemannian manifolds with integral bounds on curvature I.  Ann. Sci. \'{E}colo Norm. Sup\'{e}r. {\bf 25}(4), 77-105 (1992)

\bibitem{ZhaoW} Zhao, W.: Integral curvature bounds and diameter estimates on Finsler manifolds. Sci. China Math. {\bf 64}(3), 573-588 (2021)

\end{thebibliography}
\end{document}